\documentclass[11pt,a4paper]{article}

\usepackage[left=2.3cm, top=2.3cm,bottom=2.3cm,right=2.3cm]{geometry}

\usepackage{mathtools,amsmath,amssymb,amsthm,mathrsfs,calc,graphicx,stmaryrd,xcolor,dsfont,tikz,pgfplots,bbm,float,array,epsfig,hyperref,color}
\usetikzlibrary{calc}
\usepackage[numbers,square]{natbib}
\usepackage[british]{babel}
\usepackage{amsfonts}              
\usepackage{enumitem}
\numberwithin{equation}{section}
\numberwithin{figure}{section}

\allowdisplaybreaks[4]

\newcommand{\eqdef}{\stackrel{{\text{def}}}{=}}

\newtheorem {theorem}{Theorem}[section]
\newtheorem {proposition}[theorem]{Proposition}
\newtheorem {lemma}[theorem]{Lemma}
\newtheorem {corollary}[theorem]{Corollary}

{\theoremstyle{definition}
\newtheorem{definition}[theorem]{Definition}
\newtheorem*{convention*}{Convention}

\newtheorem {example}[theorem]{Example}

}

{
\theoremstyle{remark}
}

\newcommand{\eqdistr}{\stackrel{d}{=}}
\newcommand{\Vol}{\operatorname{Vol}}

\newcommand{\Cov}{\operatorname{Cov}}

\renewcommand{\Re}{\operatorname{Re}}  
\renewcommand{\Im}{\operatorname{Im}}  

\newcommand{\ii}{{\rm{i}}}
\newcommand{\erf}{\operatorname{erf}}

\def\ba{\begin{array}}
\def\ea{\end{array}}
\def\bea{\begin{eqnarray} \label}
\def\eea{\end{eqnarray}}
\def\be{\begin{equation} \label}
\def\ee{\end{equation}}
\def\bit{\begin{itemize}}
\def\eit{\end{itemize}}
\def\ben{\begin{enumerate}}
\def\een{\end{enumerate}}

\def\lan{\langle}
\def\ran{\rangle}

\def\BB{\mathbb{B}}
\def\CC{\mathbb{C}}

\def\HH{\mathbb{H}}

\def\N{\mathbb{N}}
\def\P{\mathbb{P}}

\def\R{\mathbb{R}}
\def\RRd1{\mathbb{R}^{d+1}}
\def\SS{\mathbb{S}}

\def\dint{\textup{d}}

\newcommand{\eee}{{\rm e}}

\newcommand{\eps}{\varepsilon}
\newcommand{\pos}{\mathop{\mathrm{pos}}\nolimits}

\newcommand{\dd}{{\rm d}}
\newcommand{\bsl}{\backslash}

\DeclareMathOperator{\arccosh}{arccosh}

\def\sgn{\mathrm{sgn}}

\begin{document}

\title{\bfseries Volumes of Regular Hyperbolic Simplices}

\author{Zakhar Kabluchko and Philipp Schange}

\date{}

\maketitle

\begin{abstract}
We derive an explicit formula for the volume of a regular simplex in the hyperbolic space of any dimension.
\end{abstract}


\section{Introduction and main result}
\subsection{Introduction}
In the Klein model of hyperbolic (Lobachevsky) geometry, the underlying space is the open $d$-dimensional unit ball $\BB^d := \{x\in \R^d: \|x\| <1\}$ endowed with the Riemannian metric
$$
\dd s^2 = \frac{(1-\|x\|^2) \|\dd x\|^2 + \langle x, \dd x \rangle^2}{(1-\|x\|^2)^2},
$$
which has constant sectional curvature $\kappa = -1$. The hyperbolic volume of a Borel set $A\subseteq \BB^d$ is
$$
\Vol_{d, -1}(A) \eqdef \int_A \frac{\dd y}{(1 - \| y \|^2)^{(d+1)/2}}.
$$

In the Klein model, hyperbolic affine subspaces are represented by the ordinary affine subspaces (intersected with $\BB^d$). A \emph{hyperbolic simplex} is just an ordinary simplex $[v_0,\ldots, v_d]$ contained in $\bar{\BB}^d:= \{x\in \R^d: \|x\| \leq 1\}$, the closed unit ball. Here, $v_0,\ldots, v_d\in \bar{\BB}^d$ are the vertices of the simplex and $[\ldots]$ denotes the convex hull.  A hyperbolic simplex is called \emph{ideal} if all of its vertices are located on the unit sphere $\SS^{d-1}$, the boundary of $\BB^d$. A hyperbolic simplex $[v_0,\ldots, v_d]$ is called \emph{regular} if any permutation of its vertices is induced by a hyperbolic isometry of $\BB^d$. For a simplex with vertices in $\BB^d$ this means that the hyperbolic distances between any two different vertices are equal to the same number $\ell>0$, the side length of the simplex. Let $\Delta_\ell^d$ be any such regular hyperbolic simplex; any other regular hyperbolic simplex with the same side length $\ell$ can be mapped to $\Delta_\ell^d$  by a hyperbolic isometry.  An \emph{ideal regular hyperbolic} simplex  $\Delta_\infty^d$ is just an ordinary regular simplex whose vertices are located on the unit sphere $\SS^{d-1}$.  Any two ideal regular hyperbolic simplices can be mapped to each other by a hyperbolic isometry.

The hyperbolic volume of an ideal regular hyperbolic  simplex $\Delta_\infty^d$ is known to be finite, for $d\geq 2$.  The following explicit values are known, see~\cite{haagerup_munkholm_simplices_max_vol_hyperbolic}, \cite[p.~207]{milnor_hyperbolic_volumes} or~\cite[Equations~(13.55),(13.56)]{kellerhals_dilogs_vol_hyperbolic_polytopes}:
$$
\Vol_{d,-1}(\Delta_\infty^d)
=
\begin{cases}
\pi, &\text{ if } d=2,\\
-3\int_{0}^{\pi/3}\log (2\sin \theta) \dint \theta = 1.01494\ldots, &\text{ if } d=3,\\
\frac{10 \pi}{3} \arcsin \frac 13 - \frac{\pi^2}{3} = 0.26889\ldots, &\text{ if } d=4.
\end{cases}
$$
The case $d=2$ is a consequence of the Gauss--Bonnet theorem, the case $d=3$ is due to Lobachevsky (who proved a formula for the volume of any ideal $3$-dimensional hyperbolic simplex, see~\cite[p.~200]{milnor_hyperbolic_volumes}), while the case $d=4$ was stated by Haagerup and Munkholm~\cite{haagerup_munkholm_simplices_max_vol_hyperbolic} (see also Exercise~11.4.3 on p.~541 in~\cite{ratcliffe_book}). For $d=5$ and $d=6$,  the values of $\Vol_{d,-1}(\Delta_\infty^d)$ can be expressed in terms of polylogarithms; see~\cite{kellerhals_vols_in_hyperbolic_5_space} and~\cite[Equation~(13.58)]{kellerhals_dilogs_vol_hyperbolic_polytopes}.
 An infinite series formula for $\Vol_{d,-1}(\Delta_\infty^d)$ is derived in~\cite[\S~4]{milnor_hyperbolic_volumes}; see also~\cite{aomoto,ribando} for more general formulas of this type.  Numerical values for $\Vol_{d,-1}(\Delta_\infty^d)$ in dimensions  $d\leq 10$ are given in~\cite[p.~207]{milnor_hyperbolic_volumes}.
The principal result of~\citet{haagerup_munkholm_simplices_max_vol_hyperbolic} is that  the ideal regular simplices  have the maximal volume among all hyperbolic simplices, and that there are no other maximizers. This result is also included in~\cite[\S~11.4]{ratcliffe_book}.  Furthermore, Haagerup and Munkholm~\cite{haagerup_munkholm_simplices_max_vol_hyperbolic} showed that  $\Vol_{d,-1}(\Delta_\infty^d) =(\eee+o(1)) \sqrt d/ d!$, as $d\to\infty$. The volume of $\Delta_\infty^d$ appears in various contexts~\cite{boeroeczky_packing_spaces_const_curv,kellerhals_ball_packing,kellerhals_reg_simpl_vol_bounds}, for example in the theory of ball packings in the hyperbolic space.


Important references on volumes of polytopes in hyperbolic and spherical space are,  besides the classical works of Lobachevsky, Bolyai and Schl\"afli, the review papers of~\citet{milnor_hyperbolic_volumes,milnor_schlafli}, \citet{vinberg_volumes_polyhedra_russ_math_surv}, \citet{abrosimov_mednykh_volumes_polytopes_spaces_const_curv,abrosimov_mednykh}, \citet{kellerhals_dilogs_vol_hyperbolic_polytopes}, the books by~\citet[Chapter~7]{alekseevski_vinberg_solodovnikov_book}, \citet{boehm_hertel_book}, \citet[Chapters~10,11]{ratcliffe_book}, and the papers by~\citet{coxeter_the_functions_of,coxeter_upper_bound_nonoverlap_spheres} and~\citet{aomoto}. There exist several explicit (but complicated) formulas for the volume of an arbitrary hyperbolic tetrahedron; see~\cite{abrosimov_mednykh_volumes_polytopes_spaces_const_curv,abrosimov_mednykh,Murakami_Yano_2005,murakami_ushijima_vol_hyperbolic_tetrahedra,abrosimov2021explicitvolumeformulahyperbolic}.  In particular, it is known~\cite[Theorem~1]{abrosimov2017} that the volume of a regular hyperbolic tetrahedron with side length $\ell$ is given by
$$
\Vol_{3,-1}(\Delta_{\ell}^3)
=
\int_0^{\ell} \frac{3a \sinh a \,  \dd a}{(1+2 \cosh a) \sqrt{(1+\cosh a)(1+3\cosh a)}} .
$$

The volumes of regular \emph{spherical} simplices in arbitrary dimension have been computed by~\citet[Section~4]{rogers} and~\citet[Lemma~4]{vershik_sporyshev}. For some probabilistic applications, see~\cite{kabluchko_zaporozhets_absorption}. A more general formula, valid for orthocentric spherical simplices, has been derived in~\cite{kabluchko_schange_angles_orthocentric_simplices}. The hyperbolic volume of the regular cube has been determined in~\cite{marshall_vol_hyperbolic_cube,smith_vol_hyperbolic_cube}.

\subsection{Volumes of regular hyperbolic simplices}
The aim of the present article is to derive an explicit formula for the volume of regular hyperbolic simplices  which is valid in any dimension $d$. The result will be stated in terms of the distribution function of the standard normal distribution, extended to complex values of the argument and  denoted by
\begin{equation}\label{eq:compl_std_normal_distr}
	\Phi: \CC \rightarrow \CC, \quad \Phi(z) \eqdef \frac{1}{2}+\frac{1}{\sqrt{2 \pi}} \int_{0}^{z} \eee^{-x^2/2} \dd x = \frac 12 + \frac 12 \erf\left(\frac{z}{\sqrt 2}\right).
\end{equation}
Here,  $\erf(y)$
denotes the error function  whose properties are documented in~\cite[Chapter~7]{NIST:DLMF}; see also~\cite[Chapter~8]{NIST:DLMF} for the closely related incomplete Gamma function.
The integral in~\eqref{eq:compl_std_normal_distr} is taken along any contour in the complex plane $\CC$ connecting $0$ to $z$.
It is well known that  $\Phi(z)$ is an entire function.

\begin{theorem}[Volume of a regular hyperbolic simplex]\label{theo:vol_regular_hyperbolic_simplex_intro}
Let $d\geq 2$.  In the $d$-dimensional hyperbolic space of constant curvature $\kappa = -1$ consider a regular $d$-dimensional hyperbolic simplex $\Delta_\ell^d$  with hyperbolic side length $\ell>0$.  Then, the hyperbolic volume of $\Delta_\ell^d$ is given by
\begin{equation*}
\frac{\sqrt 2 \, \pi^{d/2}} {\ii^d \, \Gamma(\frac{d+1}2)}
\int_{0}^{(1-\ii)\infty} \left(\Phi^{d+1} \left(\ii\, \sqrt{\frac{\cosh \ell}{1+d \cosh \ell}} \, x\right) +  \Phi^{d+1} \left(-\ii\, \sqrt{\frac{\cosh \ell}{1+d \cosh \ell}} \,  x \right)\right) \eee^{-\frac{x^2}2} \dd x.
\end{equation*}
The hyperbolic volume of the ideal  regular $d$-dimensional hyperbolic simplex $\Delta_\infty^d$ is given by
\begin{equation*}
\frac{\sqrt 2 \, \pi^{d/2}} {\ii^d \, \Gamma(\frac{d+1}2)}
\int_{0}^{(1-\ii)\infty} \left(\Phi^{d+1} \left(\frac{\ii x}{\sqrt d} \right) +  \Phi^{d+1} \left(-\frac{\ii x}{\sqrt d}\right) \right)  \eee^{-\frac{x^2}2} \dd x.
\end{equation*}
\end{theorem}

Here, $\ii= \sqrt{-1}$ and $\Gamma(\cdot)$ is the Gamma function.
Integrals of the form $\int_0^{\omega \infty} f(x) \dd x$, where $\omega \in \CC\backslash\{0\}$ and $f:\CC\to\CC$ is an entire function, are defined as improper integrals by the convention
$$
\int_0^{\omega \infty} f(x) \dd x \eqdef \lim_{B\to +\infty} \int_0^B f(\omega y) \omega \dd y,
$$
whenever the limit exists. The fact that the improper integrals appearing in Theorem~\ref{theo:vol_regular_hyperbolic_simplex_intro} are convergent is non-trivial and will be established below.

Our main result is more general than Theorem~\ref{theo:vol_regular_hyperbolic_simplex_intro}.
In fact, we establish a formula for the hyperbolic and spherical volumes of orthocentric simplices, with regular simplices arising as a special case.
This general theorem will be stated as Theorem~\ref{theo:vol_acute_orthoc_simpl_in_Klein_model_any_dimension_1-i} after the necessary preparations.


\section{Notation and facts from hyperbolic and spherical geometry}\label{sec:notation_facts_hyperbolic_geom}
\paragraph{Basic notation.}
Let $\langle \cdot,  \cdot\rangle$ denote the Euclidean scalar product and $\|\cdot\|$ the Euclidean norm.
Let $\BB^d = \{x\in \R^d: \|x\| < 1\}$ be the open unit ball in $\R^d$  and $\bar \BB^d = \{x\in \R^d: \|x\| \leq  1\}$ its closure. The $d$-dimensional unit sphere in $\R^{d+1}$ is denoted by $\SS^{d}=\{x\in\R^{d+1}:\|x\|=1\}$. Its $d$-dimensional surface area is denoted by
$$
\omega_{d+1}:= \frac{2\pi^{(d+1)/2}}{\Gamma((d+1)/2)}.
$$

The variable $\kappa\in \R$ always denotes the curvature.  In the following, we recall some basic facts from hyperbolic ($\kappa<0$) and spherical ($\kappa>0$) geometry. For a comprehensive treatment of the subject we refer to~\cite{alekseevski_vinberg_solodovnikov_book} and~\cite{ratcliffe_book}.

\paragraph{Upper half-sphere.}
For $\kappa > 0$,  the upper half-sphere of radius $\tfrac1{\sqrt{\kappa}}$ in $\R^{d+1}$ is denoted by
$$
\SS^d_+ (\kappa) :=\left\{ x= (x_0,x_1,\ldots, x_d)\in \R^{d+1}:\|x\|^2=\frac{1}{\kappa},x_0>0 \right\}.
$$
The restriction of the Euclidean scalar product in $\R^{d+1}$ to tangent spaces of $\SS^d_+(\kappa)$ defines  a Riemannian metric on $\SS^d_+(\kappa)$ with constant curvature $\kappa$.

\paragraph{Upper hyperboloid.}
Let now  $\kappa < 0$.
The Minkowski product of $x=(x_0,x_1,\ldots, x_d) \in \R^{d+1}$ and $y=(y_0,y_1,\ldots, y_d) \in \R^{d+1}$ is defined as $B(x,y)=x_0y_0 -(x_1y_1 +\ldots +x_dy_d)$. Consider the upper hyperboloid
$$
\HH^d_+(\kappa) := \left\{x = (x_0,x_1,\ldots, x_d) \in \R^{d+1}:B(x,x) = -\frac{1}{\kappa}, x_0>0 \right\}.
$$
The restriction of the Minkowski product to tangent spaces of $\HH^d_+(\kappa)$ defines a Riemannian metric on $\HH^d_+(\kappa)$ with constant curvature $\kappa$.

\paragraph{Gnomonic projection.}
The \emph{gnomonic projection} is defined as follows:
\begin{align*}
\text{For $\kappa >0$}: \qquad &\Pi_{gn}:\SS^d_+ (\kappa) \rightarrow \R^d, \qquad  (x_0,x_1,\ldots, x_d) \mapsto \frac 1 {\sqrt \kappa} \left(\frac{x_1}{x_0},\ldots, \frac{x_d}{x_0}\right)\in \R^d,
\\
\text{For $\kappa <0$}: \qquad & \Pi_{gn}: \HH^d_+(\kappa) \rightarrow \frac{1}{\sqrt{-\kappa}}\BB^d, \qquad (x_0,x_1,\ldots, x_d) \mapsto \frac 1 {\sqrt {-\kappa}} \left(\frac{x_1}{x_0},\ldots, \frac{x_d}{x_0}\right) \in \frac{1}{\sqrt{-\kappa}}\BB^d.
\end{align*}
For an illustration, see Figure~\ref{figure:gnomonic_projection}. For $\kappa = 0$, the gnominic projection $\Pi_{gn}:\R^d\to \R^d$ is the identity map.
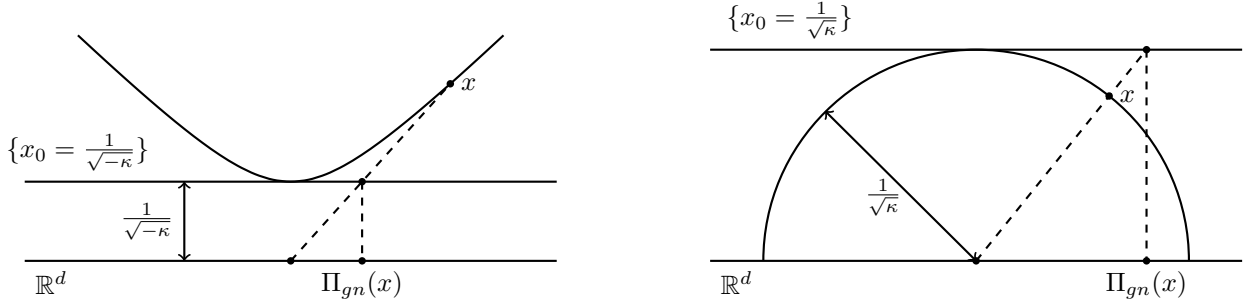
\begin{figure}
	\centering
	\begin{tikzpicture}[scale=0.7,thick,help lines/.style={draw=black!50}]	
		\draw[domain=-4:4,smooth,variable=\x] plot ({\x},{sqrt(\x*\x+2.25)});
		
		\coordinate (A) at (0,0);
		\draw (-5,1.5)--(5,1.5);
		\draw (-5,0)--(5,0);
		\draw [<->] (-2,0)--(-2,1.5);
		\coordinate[label=left: \small ${\frac{1}{\sqrt{-\kappa}}}$] (h1) at (-2,0.75);
		\coordinate[label=below: \small $\R^d$] (h2) at (-4.5,0);
		\coordinate[label=above: \small {$\{x_0=\frac{1}{\sqrt{-\kappa}}\}$}] (h3) at (-4,1.5);
		
		\coordinate[label=right:\small $x$] (x) at (3,3.354);
		\coordinate (p1) at (1.342,1.5);
		\coordinate[label=below:\small 	{$\Pi_{gn}(x)$}] (u) at (1.342,0);
		
		\draw[dashed] (A)--(x);
		\draw[dashed] (p1)--(u);
		
		\foreach \point in {A,x,p1,u} \fill [black] (\point) circle (2pt);
	\end{tikzpicture}
	\hfill
	\begin{tikzpicture}[scale=0.7,thick,help lines/.style={draw=black!50}]
		\coordinate (A) at ($ (0,0) $);
		\coordinate (B) at ($ (4,0) $);
		\coordinate [label=left: \small ${\frac{1}{\sqrt{\kappa}}}$] (C) at (-1.2,1.2);
		\draw (B) let \p1 = ($ (A) - (B) $) in arc (0:180:{veclen(0,4)});
		\draw [<->] (A)--(-2.84,2.84);
		
		\draw (-5,4)--(5, 4);
		\coordinate[label=above: \small ${ \{x_0=\frac{1}{\sqrt{\kappa}}\} }$] (E) at (-3.5, 4);
		\draw (-5,0)--(5,0);
		\coordinate[label=below: \small $\R^d$] (h2) at (-4.5,0);
		
		\coordinate[label=right:\small $x$] (x) at (2.5, 3.1225);
		\coordinate (proj) at (3.203,4);
		\coordinate[label=below: \small 	{$\Pi_{gn}(x)$}] (u) at (3.203,0);
		
		\draw[dashed] (0,0)--(proj);
		\draw[dashed] (proj)--(u);
		
		\foreach \point in {A,x,proj,u} \fill [black] (\point) circle (2pt);
	\end{tikzpicture}
	\caption{Gnomonic projection: Hyperbolic case (left) and spherical case (right)}
	\label{figure:gnomonic_projection}
\end{figure}

\paragraph{Klein model.}
The target space of the gnomonic projection is the underlying space of the Klein model, which will be denoted by
\begin{align*}
\BB^d(\kappa):=
\begin{cases}
\R^d, & \text{ if } \kappa\geq 0,\\
\frac{1}{\sqrt{-\kappa}}\BB^d = \left\{x\in \R^d: \|x\| < \tfrac 1 {\sqrt{-\kappa}}\right\}, & \text{ if } \kappa<0.
\end{cases}
\end{align*}
The push-forward of the Riemannian metric on $\SS^d_+(\kappa)$ (if $\kappa >0$), $\HH^d_+(\kappa)$ (if $\kappa <0$) or the Euclidean metric on $\R^d$ (if $\kappa = 0$), defines a Riemannian metric on $\BB^d(\kappa)$ given by
\begin{equation} \label{eq:thm_metric_g_ij}
\dd s^2 = \frac{(1+\kappa \|u\|^2)\|\dd u\|^2 - \kappa \langle u, \dd u\rangle^2}{(1+\kappa \|u\|^2)^2}.
\end{equation}
The hyperbolic (if $\kappa <0$),  spherical (if $\kappa >0$) or Euclidean (if $\kappa = 0$) volume  of a Borel-measurable set $A \subseteq \BB^d(\kappa)$ is given by
\begin{equation} \label{eq:thm_volume}
		\Vol_{d, \kappa}(A) = \int_A \frac{\dd y}{(1+\kappa \| y \|^2)^{(d+1)/2}}.
\end{equation}
Formulas~\eqref{eq:thm_metric_g_ij} and~\eqref{eq:thm_volume} can be found, for example, in~\cite[p.~523]{ratcliffe_book}.  For $\kappa >0$,  the total spherical volume of $\BB^d(\kappa)=\R^d$ is finite and given by
$$
\Vol_{d, \kappa}(\R^d)=\frac{1}{2} \frac{\omega_{d+1}}{\kappa^{d/2}}.
$$
In the hyperbolic case, when $\kappa<0$, the closure of $\BB^d(\kappa)$ is denoted by $\bar \BB^d(\kappa) := \frac{1}{\sqrt{-\kappa}}\bar \BB^d$.
For $\kappa \geq 0$ we write $\bar \BB^d(\kappa):= \R^d$.

\section{Main result}

\subsection{The setting: Orthocentric simplices}

In the next definition we introduce the class of simplices $[v_0,\ldots,v_d]\subseteq \R^d$ for which  the spherical and hyperbolic volumes will be computed.
\begin{definition}[Vectors in orthocentric position]\label{def:orthocentric_position}
Let $d\in \N$, $\tau_0,\ldots, \tau_d >0$ and put $s= \tau_0^2 + \ldots +\tau_d^2$.  We say that vectors $v_0,\ldots, v_d \in \R^d$ are in \emph{orthocentric position} with parameters $\tau_0,\ldots, \tau_d$ if
\begin{equation}\label{eq:scalar_prod_orthocentric_vectors}
\lan v_j,v_k \ran = -\frac 1s,
\quad
0\leq j \neq k \leq d,
\quad
\|v_j\|^2
=
-\frac{1}{s} + \frac{1}{\tau_j^2},
\quad
0\leq j  \leq d.
\end{equation}
\end{definition}
\begin{example}[Regular simplices]
If $\tau_0=\ldots = \tau_d = \tau>0$, then $\|v_j-v_k\|^2 = 2/\tau^2$ for all $j\neq k$, which means that the simplex $[v_0,\ldots, v_d]$ is regular. Moreover, the center of this simplex is at $0$. To prove this, one checks that $\|v_0+\ldots+v_d\|^2 = 0$. Conversely, if $[v_0,\ldots, v_d]\subseteq \R^d$ is a regular simplex with $v_0+\ldots+v_d = 0$, then $v_0,\ldots, v_d$ satisfy~\eqref{eq:scalar_prod_orthocentric_vectors} with equal $\tau_j$'s.
\end{example}
\begin{example}[Orthocentric simplices]
To motivate Definition~\ref{def:orthocentric_position}, let $e_0,\ldots, e_d$ be the standard orthonormal basis of $\R^{d+1}$ and consider the $d$-dimensional simplex  $[\tfrac {e_0}{\tau_0}, \ldots, \tfrac {e_d}{\tau_d}]$. This simplex is orthocentric~\cite{Edmonds2005}: its altitudes intersect in the point $H := \tfrac{\tau_0 e_0 + \ldots + \tau_d e_d}{\tau_0^2+\ldots + \tau_d^2}$; see, e.g.,~\cite[Example 3.3]{kabluchko_schange_angles_orthocentric_simplices}. A computation shows that for all $j,k \in \{0,\ldots, d\}$,
$$
\left\lan \frac{e_j}{\tau_j}-H,\frac{e_k}{\tau_k} - H \right\ran
=
\left\lan \frac{e_j}{\tau_j} - \frac{\tau_0 e_0 + \ldots + \tau_d e_d}{\tau_0^2+\ldots + \tau_d^2}, \frac{e_k}{\tau_k} - \frac{\tau_0 e_0 + \ldots + \tau_d e_d}{\tau_0^2+\ldots + \tau_d^2} \right\ran
=
-\frac{1}{s} + \frac{\delta_{jk}}{\tau_j^2},
$$
where $\delta_{jk}$ is the Kronecker delta. These are the same scalar products as in~\eqref{eq:scalar_prod_orthocentric_vectors}.
Hence, the simplices $[v_0,\ldots, v_d]$ and $[\tfrac {e_0}{\tau_0}, \ldots, \tfrac {e_d}{\tau_d}]$ are isometric. Moreover, the simplex $[v_0,\ldots, v_d]$ is orthocentric: its altitudes intersect at $0$.  Although we shall not need this fact, let us mention that the following converse statement holds~\cite[Proposition~3.4]{kabluchko_schange_angles_orthocentric_simplices}: If the altitudes of a $d$-dimensional simplex $[v_0,\ldots, v_d]\subseteq \R^d$ intersect at $0$ and $0$ belongs to the interior of the simplex, then $v_0,\ldots, v_d$ satisfy~\eqref{eq:scalar_prod_orthocentric_vectors} for some $\tau_0,\ldots, \tau_d>0$.
\end{example}

Our aim is to compute the spherical or hyperbolic volume $\Vol_{d,\kappa}(Q)$ of the simplex $Q:= [v_0,\ldots, v_d]$, where $v_0,\ldots, v_d$ are as in Definition~\ref{def:orthocentric_position}. Of course, this is only possible if $Q \subseteq \bar \BB^d(\kappa)$.
In the next lemma we specify the range of $\kappa$ for which this condition  holds.

\begin{lemma}\label{lemma:lower_limit_kappa_simplex}
Let $v_0,\ldots, v_d\in \R^d$ be vectors as in Definition~\ref{def:orthocentric_position}.  Then, $[v_0,\ldots, v_d] \subseteq \BB^d(\kappa)$ if and only if $\kappa>\kappa_0$, and $[v_0,\ldots, v_d] \subseteq \bar \BB^d(\kappa)$ if and only if $\kappa \geq \kappa_0$, where
$$
\kappa_0
=
-\min_{0\leq j \leq d}\frac{\tau_j^2 s}{s- \tau_{j}^2} = -\min_{0\leq j \leq d} \frac{1}{1/\tau_j^2 - 1/s} <0.
$$
If $\tau_0= \min \{\tau_0,\ldots, \tau_d\}$, then $\kappa_0 =-\frac{\tau_0^2 s}{s-\tau_0^2}$.
\end{lemma}
\begin{proof}
For $\kappa \geq 0$, we have $\BB^d(\kappa) = \R^d$ and the condition $[v_0,\ldots, v_d] \subseteq \BB^d(\kappa)$ is satisfied.
It remains to consider $\kappa<0$. In this case, the Klein model is   $\BB^d(\kappa) = \frac{1}{\sqrt{-\kappa}}\BB^d$.  Then, $[v_0,\ldots, v_d] \subseteq \BB^d(\kappa)$ if and only if $\|v_j\|^2 < -1/\kappa$ for every $j\in \{0,\ldots, d\}$.  Using~\eqref{eq:scalar_prod_orthocentric_vectors}, this can be written as $-\frac{1}{s} + \frac{1}{\tau_j^2}<-\frac{1}{\kappa}$ for all $j\in \{0,\ldots, d\}$, which yields the claim.
\end{proof}

In the next proposition, we compute the \emph{hyperbolic} side lengths of the simplex $[v_0,\ldots, v_d]$.
\begin{proposition}\label{prop:side_lengths_hyperbolic_orthocentric_simplex}
Let $v_0,\ldots, v_d\in \R^d$ be vectors as in Definition~\ref{def:orthocentric_position} and $\kappa \in (\kappa_0, 0)$.  Then, for all $0\leq j\neq k \leq d$, the hyperbolic distance between $v_j$ and $v_k$ in the Klein model $\BB^d(\kappa)$ is given by
$$
\ell_{j,k} = \frac{1}{\sqrt{-\kappa}} \arccosh \left( \frac{s-\kappa}{\sqrt{s-\kappa+\kappa \frac{s}{\tau_j^2}} \sqrt{s - \kappa +\kappa \frac{s}{\tau_k^2}}} \right).
$$
\end{proposition}
\begin{proof}
Let $\Pi_{gn}^{-1}: \BB^d(\kappa) \to \HH^d_+(\kappa)$ be the inverse of the gnomonic projection; see Section~\ref{sec:notation_facts_hyperbolic_geom}.  The hyperbolic distance between $v_j\in\BB^d (\kappa)$ and $v_k\in\BB^d (\kappa)$ is
\begin{equation*}
\ell_{j,k}
=
\frac{1}{\sqrt{-\kappa}} \arccosh \left( -\kappa B\left(\Pi_{gn}^{-1}(v_j) , \Pi_{gn}^{-1}(v_k)\right) \right).
\end{equation*}
Let $e_0, e_1, \ldots, e_d$ be the standard orthonormal basis of $\R^{d+1}$. We identify the space $\R^d$ in which $Q$ is contained with the linear hull of $e_1,\ldots, e_d$.
The definition of the gnomonic projection yields that $\Pi_{gn}^{-1} v_j = (v_j + e_0/\sqrt{-\kappa})/\sqrt{1+\kappa \|v_j\|^2}\in \HH^d_+(\kappa)$ for all $j\in \{0, \ldots, d\}$. Hence,
\begin{align*}
\ell_{j,k}
&=
\frac 1{\sqrt{-\kappa}} \arccosh \left( -\kappa B\left(\frac{v_j + e_0/\sqrt{-\kappa}}{\sqrt{1+\kappa \|v_j\|^2}} , \frac{v_k + e_0/\sqrt{-\kappa}}{\sqrt{1+\kappa \|v_k\|^2}}\right) \right)
\\
&=
\frac{1}{\sqrt{-\kappa}} \arccosh\left( -\kappa \frac{\frac{1}{-\kappa} - \lan v_j,v_k \ran}{\sqrt{1+\kappa \|v_j\|^2} \sqrt{1+\kappa \|v_k\|^2}} \right).
\end{align*}
Plugging~\eqref{eq:scalar_prod_orthocentric_vectors}  into this formula completes the proof.
\end{proof}

\subsection{Statement of the main result: Hyperbolic volume of orthocentric simplices}
The next theorem is our main result. It gives an explicit formula for the volume of an orthocentric simplex in a spherical or hyperbolic geometry of constant curvature $\kappa$.
\begin{theorem}[Volume of an orthocentric simplex, spherical or hyperbolic]
\label{theo:vol_acute_orthoc_simpl_in_Klein_model_any_dimension_1-i}
Fix $d\geq 2$, $\tau_0,\ldots, \tau_d>0$ and define $s=\tau_0^2 +\ldots + \tau_d^2$. Consider a simplex $Q := [v_0, \ldots, v_d]$, where  $v_0,\ldots, v_d\in \R^d$ are vectors in orthocentric position with parameters $\tau_0,\ldots, \tau_d$; see Definition~\ref{def:orthocentric_position}. Let $\kappa \neq 0$ be such that $Q\subseteq \bar{\BB}^d(\kappa)$; see Lemma~\ref{lemma:lower_limit_kappa_simplex}.
Then,
\begin{equation}\label{eq:vol_acute_orthoc_simpl_in_Klein_model_even_and_odd_upper}
\Vol_{d,\kappa}(Q)
=
\frac{\omega_{d+1}}{\sqrt{\kappa}^{d}} \frac1{\sqrt{2\pi}} \int_{0}^{(1-\ii)\infty} \left( \prod_{j=0}^d \Phi \left( \frac{\tau_j}{s} \sqrt{\kappa -s} x \right) + \prod_{j=0}^d \Phi \left(- \frac{\tau_j}{s} \sqrt{\kappa -s} x \right) \right) \eee^{-\frac{x^2}2} \dd x,
\end{equation}
with the convention that $\sqrt{-r} = \ii \sqrt r$ for $r\geq 0$.
The formula remains valid if $\int_{0}^{(1-\ii)\infty}$ is replaced by $\int_{0}^{(1+\ii)\infty}$, this time with the opposite convention  $\sqrt{-r} = -\ii \sqrt r$ for $r\geq 0$.
\end{theorem}

Before giving the proof of Theorem~\ref{theo:vol_acute_orthoc_simpl_in_Klein_model_any_dimension_1-i}, let us derive some of its corollaries. Taking $\kappa = -1$ in these corollaries gives Theorem~\ref{theo:vol_regular_hyperbolic_simplex_intro}.

\begin{theorem}[Volume of a regular hyperbolic simplex]\label{theo:vol_regular_hyperbolic_simplex}
Let $d\geq 2$ and $\kappa  < 0$. Consider a  regular $d$-dimensional hyperbolic simplex $\Delta_\ell^d(\kappa)$ in $\BB^d(\kappa)$ with hyperbolic side length $\ell\in (0,\infty)$. Then,  the hyperbolic volume of $\Delta_\ell^d(\kappa)$ is given by
\begin{equation}\label{eq:vol_regular_simplex}
\Vol_{d,\kappa}(\Delta_\ell^d(\kappa))
=
\frac{\omega_{d+1}}{\ii^d |\kappa|^{d/2}}  \frac1{\sqrt{2\pi}} \int_{0}^{(1-\ii)\infty} \left(\Phi^{d+1} (\ii \sqrt{\lambda}  x) +  \Phi^{d+1} (-\ii \sqrt{\lambda} x) \right) \eee^{-\frac{x^2}2} \dd x,
\end{equation}
where $\lambda = \cosh(\ell \sqrt{-\kappa})/(1+d \cosh(\ell\sqrt{-\kappa}))$. 
\end{theorem}

\begin{proof}
We apply Theorem~\ref{theo:vol_acute_orthoc_simpl_in_Klein_model_any_dimension_1-i} with $\tau_0 = \ldots = \tau_d = \tau>0$, where
$$
\tau^2
:=
\frac{(-\kappa) (1+d \cosh (\ell \sqrt{-\kappa}))}{(d+1)(\cosh(\ell \sqrt{-\kappa})-1)},
\qquad
s
:=
(d+1)\tau^2
=
\frac{(-\kappa) (1+d \cosh (\ell \sqrt{-\kappa} ))}{\cosh(\ell \sqrt{-\kappa})-1}.
$$
By Proposition~\ref{prop:side_lengths_hyperbolic_orthocentric_simplex}, the hyperbolic side length of the simplex $Q$ appearing in Theorem~\ref{theo:vol_acute_orthoc_simpl_in_Klein_model_any_dimension_1-i} is
\begin{equation*}
\frac 1{\sqrt{-\kappa}} \arccosh \left( \frac{(d+1)\tau^2-\kappa}{(d+1)\tau^2+\kappa d} \right) = \ell.
\end{equation*}
So, we can identify $Q$ with the regular hyperbolic simplex $\Delta_\ell^d(\kappa)$.
A computation shows that $\frac{\tau_j}{s} \sqrt {\kappa-s} = \sqrt{-\lambda}$ with  $\lambda = \frac{\cosh(\ell \sqrt{-\kappa})}{1+d \cosh(\ell\sqrt{-\kappa})}$. To complete the proof, apply Theorem~\ref{theo:vol_acute_orthoc_simpl_in_Klein_model_any_dimension_1-i}.
\end{proof}

In the Klein model $\BB^d (\kappa)$ with $\kappa <0$, the ideal regular $d$-dimensional hyperbolic simplex $\Delta_\infty^d(\kappa)$ is represented by a Euclidean regular $d$-dimensional simplex inscribed into a sphere of radius $1/\sqrt{-\kappa}$ in $\R^d$.

\begin{theorem}[Volume of an ideal regular hyperbolic simplex]\label{theo:vol_ideal_regular_hyperbolic_simplex}
Let $d\geq 2$  and  $\kappa  < 0$. 
Then, the hyperbolic volume of the ideal regular $d$-dimensional hyperbolic simplex in $\BB^d (\kappa)$ is given by
\begin{equation}\label{eq:vol_regular_ideal_hyperbolic_simplex}
\Vol_{d,\kappa}(\Delta_\infty^d(\kappa))
=
\frac{\omega_{d+1}}{\ii^d |\kappa|^{d/2}} \frac1{\sqrt{2\pi}} \int_{0}^{(1-\ii)\infty} \left(\Phi^{d+1} \left(\frac{\ii x}{\sqrt d} \right) +  \Phi^{d+1} \left(-\frac{\ii x}{\sqrt d}\right) \right)  \eee^{-\frac{x^2}2} \dd x.
\end{equation}
\end{theorem}
\begin{proof}
 The \emph{Euclidean} side length of $\Delta_\infty^d(\kappa)$  is $\sqrt {2(d+1)}/ \sqrt{-\kappa d}$. Applying Theorem~\ref{theo:vol_acute_orthoc_simpl_in_Klein_model_any_dimension_1-i} with $\tau_0 = \ldots = \tau_d = \sqrt {-\kappa d}/\sqrt{d+1}$ and $s=-\kappa d$ gives the claimed formula.
\end{proof}

\section{Proof of the main result}
\subsection{Outline of the proof}
%
%
The proof of Theorem~\ref{theo:vol_acute_orthoc_simpl_in_Klein_model_any_dimension_1-i} proceeds in the following steps.
\begin{enumerate}[itemsep=0pt,label=(\arabic*)]
\item Show that the right-hand side of the formula
$
\Vol_{d, \kappa}(Q) \;=\; \int_Q \frac{\dd y}{(1+\kappa \| y \|^2)^{(d+1)/2}}
$
is analytic in $\kappa$ on the domain $\CC \bsl (-\infty, \kappa_0]$.
This will be established in Section~\ref{subsec:anal_cont_volume}.

\item Derive a formula for $\Vol_{d, \kappa}(Q)$ that holds for $\kappa \geq s$.
The normal distribution function $\Phi$ naturally appears in this formula through a probabilistic interpretation of the spherical volume.
This step will be carried out in Section~\ref{subsec:spherical_volume_orthocentric_simpl_kappa_geq_s}.

\item Construct an analytic continuation of the expression obtained in the previous step to the domain $\CC \bsl (-\infty,0]$.
This constitutes the most nontrivial step and represents the main contribution of the present work; see Section~\ref{subsec:analytic_continuation_integral_Phi}.

\item Apply the uniqueness theorem for analytic functions to conclude the proof of Theorem~\ref{theo:vol_acute_orthoc_simpl_in_Klein_model_any_dimension_1-i}; see Section~\ref{subsec:proof_formula_hyperbolic_volume_complete}.
\end{enumerate}

\subsection{Analytic continuation of the volume}\label{subsec:anal_cont_volume}
\begin{lemma}\label{lemma:Vol_d_kappa_analytic}
For $d\in \N$, let $A \subseteq \R^d$ be a compact set.
Let
$
\kappa_0 := \inf\{\kappa \in \R: A \subseteq \BB^d(\kappa)\}.
$
(Clearly, $\kappa_0 <0$.) Then, the function
$$
V(\kappa) := \int_A \frac{\dd y}{(1+\kappa \| y \|^2)^{(d+1)/2}}
$$
is well defined and analytic in $\kappa$ on the domain $\CC \backslash (-\infty, \kappa_0]$.
\end{lemma}
\begin{proof}
By Morera's theorem, it suffices to show that $V(\kappa)$ is continuous on $\CC \backslash (-\infty, \kappa_0]$ and that its contour integrals  over all triangular paths contained in $\CC \backslash (-\infty, \kappa_0]$ vanish.

\vspace*{2mm}
\noindent
\emph{Step 1: Preliminaries.}
As a preliminary step, observe that for $y \in A$ and $\kappa \in \CC \backslash (-\infty, \kappa_0]$ we have $1+\kappa \|y\|^2 \notin (-\infty, 0]$. In particular, the function $\kappa \mapsto (1 + \kappa \|y\|^2)^{-(d+1)/2}$ is analytic on $\CC \backslash (-\infty, \kappa_0]$ for every $y\in A$.  Indeed, the contrapositive holds: If, for some $\kappa \in \CC$ and $y \in A$,  we have
$
1+\kappa \|y\|^2 \in (-\infty, 0],
$
then necessarily $\kappa\in \R$, $\kappa <0$, and  $\|y\|^2 \in [-1/\kappa, \infty)$, which implies
$
A \not\subseteq \tfrac{1}{\sqrt{-\kappa}}\, \BB^d = \BB^d(\kappa).
$
By the definition of $\kappa_0$, this yields $\kappa \in (-\infty, \kappa_0]$.

Another observation we need is that for every compact set $K \subseteq \CC \backslash (-\infty, \kappa_0]$ there exists $c_K>0$ such that
$$
\bigl|1+\kappa \|y\|^2\bigr| \ge c_K
\quad \text{for all } \kappa \in K,\; y \in A.
$$
Indeed, the map
$
(\kappa,y) \mapsto 1+\kappa\|y\|^2
$
is continuous on the compact set $K \times A$ and, by the preceding observation, its image avoids $0$.

\vspace*{2mm}
\noindent
\emph{Step 2: Continuity.} To prove that $V(\kappa)$
is continuous on $\CC \backslash (-\infty, \kappa_0]$, take some $\kappa, \kappa_1, \kappa_2, \ldots \in \CC \backslash (-\infty, \kappa_0]$ with $\lim_{n\to\infty}\kappa_n = \kappa$.
By Step~1, there exists $c>0$ such that
$
|1+\kappa_n \|y\|^2|  \ge c
$
for all
$n \in \N$ and  $y \in A$.
Thus, the  dominated convergence theorem applies and yields
$$
\lim_{n \to \infty} \int_A \frac{\dd y}{(1+\kappa_n \|y\|^2)^{(d+1)/2}}
=
\int_A \frac{\dd y}{(1+\kappa \|y\|^2)^{(d+1)/2}},
$$
which  proves the continuity of $V$.

\vspace*{2mm}
\noindent
\emph{Step 3: Contour integrals.}
Let $T \subseteq \CC \backslash (-\infty,\kappa_0]$ be a closed triangle and $\partial T$ be the corresponding (oriented) triangular contour. Our aim is to show that $\oint_{\partial T} V(\kappa) \dd \kappa = 0$.  Let $\gamma: [0,1] \to \partial T$ be a piecewise smooth parametrization of $\partial T$.
As observed in Step 1, for each fixed $y\in A$, the function
$
f_y(\kappa):= (1+\kappa\|y\|^2)^{-(d+1)/2}
$
is analytic on $\CC \backslash (-\infty,\kappa_0]$. By Cauchy's theorem,
$$
\oint_{\partial T} f_y(\kappa) \,\dd \kappa
\eqdef
\int_0^1 f_y(\gamma(t))\gamma'(t) \dd t
=
0.
$$
We now want to integrate this identity over $A$ and interchange the integrals. By Step~1 of the proof, there exists $c>0$ such that $|1+\kappa\|y\|^2|  \geq c$ for all $\kappa\in T,\ y\in A$.
It follows that for every $t\in[0,1]$ (with an exception of a finite set) and every  $y\in A$,
$$
\bigl|f_y(\gamma(t))\gamma'(t)\bigr|
=
|\gamma'(t)| \bigl |1+\kappa\|y\|^2 \bigr|^{-(d+1)/2}
\leq |\gamma'(t)|\,c^{-(d+1)/2} \leq C.
$$
So, the function $(t,y)\mapsto |f_y(\gamma(t))\gamma'(t)|$ is bounded on $[0,1] \times A$. By Fubini's theorem, we interchange the integrals to obtain
$$
\oint_{\partial T} V(\kappa) \dd \kappa = \int_{0}^1 \left(\int_A f_{y}(\gamma(t))\,\dd y\right) \gamma'(t) \dd t
=
\int_A\! \left(\int_{0}^1 f_y(\gamma(t)) \gamma'(t) \, \dd t\right) \dd y
=
\int_A 0 \, \dd y
=
0.
$$
This shows that  the integral of $V(\kappa)$ along any triangular contour in $\CC \backslash (-\infty,\kappa_0]$ vanishes.

\vspace*{2mm}
Thus, Morera's theorem applies and shows that $V(\kappa)$
is analytic on $\CC \backslash (-\infty, \kappa_0]$.
\end{proof}

\subsection{Spherical volume for \texorpdfstring{$\kappa \geq s$}{kappa >= s}} \label{subsec:spherical_volume_orthocentric_simpl_kappa_geq_s}
The upcoming proposition provides a formula for the spherical volume of an orthocentric simplex for $\kappa \geq s$.

\begin{proposition}[Spherical volume for  $\kappa \geq s$]\label{prop:vol_spher_orthocentr_simplex_kappa_greater_s}
Let $d\geq 2$, $\tau_0,\ldots, \tau_d>0$ and put $s=\tau_0^2 + \ldots + \tau_d^2$. Consider vectors $v_0,\ldots, v_d\in  \R^d$ that are in orthocentric position  with parameters $\tau_0,\ldots, \tau_d$; see Definition~\ref{def:orthocentric_position}. Consider the simplex $Q := [v_0, \ldots, v_d]\subseteq \R^d$. Then, for all $\kappa\geq s$ we have
$$
\Vol_{d,\kappa}(Q)
=
\frac{\omega_{d+1}}{\sqrt{\kappa}^{d}} \frac1{\sqrt{2\pi}} \int_{0}^{+\infty} \left( \prod_{j=0}^d \Phi \left( \frac{\tau_j}{s} \sqrt{\kappa -s} x \right) + \prod_{j=0}^d \Phi \left(- \frac{\tau_j}{s} \sqrt{\kappa -s} x \right) \right) \eee^{-\frac{x^2}2} \dd x.
$$
\end{proposition}

\begin{proof}
The proof could be extracted from~\cite{kabluchko_schange_angles_orthocentric_simplices}, but to make the paper self-contained we provide a full argument.
Let $e_0, e_1, \ldots, e_d$ be the standard orthonormal basis of $\R^{d+1}$ and identify the space $\R^d$ in which $Q$ is located with the linear hull of $e_1,\ldots, e_d$. Consider the vectors
$$
x_j = \tfrac{e_0}{\sqrt{\kappa}} +v_j \in \R^{d+1},
\qquad j\in \{0, \ldots, d\},
$$
and let $\pos (x_0,\ldots, x_d)$ be the positive hull of these vectors, i.e.\ the set of linear combinations $\lambda_0 x_0 + \ldots + \lambda_d x_d$ with $\lambda_0\geq 0,\ldots, \lambda_d\geq 0$. Let $N= (g_0,\ldots, g_d)$ be a $(d+1)$-dimensional multivariate standard Gaussian  random vector, that is, a vector whose components $g_0,\ldots, g_d$ are  independent standard Gaussian random variables. Since the distribution of $N$ is invariant w.r.t.\ the orthogonal transformations of $\R^{d+1}$, the random vector $N/\|N\|$ is uniformly distributed on $\SS^d$. It follows that the spherical volume of $Q$ is given by
$$
\Vol_{d, \kappa}(Q)
=
\frac{\omega_{d+1}}{\sqrt{\kappa}^d}
\P\left[\frac N{\|N\|} \in \pos(x_0, \ldots, x_d) \right]
=
\frac{\omega_{d+1}}{\sqrt{\kappa}^d}
\P\left[N \in \pos(x_0, \ldots, x_d) \right]
.
$$

The dual cone of $\pos (x_0,\ldots, x_d)$ is defined as the set of all $y\in \R^{d+1}$ with the property $\langle y, x\rangle \leq 0$ for all $x\in  \pos (x_0,\ldots, x_d)$. Next, our aim is to show that the dual cone of $\pos (x_0,\ldots, x_d)$ can be represented as $\pos (y_0,\ldots, y_d)$ for some vectors $y_0,\ldots, y_d \in \R^{d+1}$ with scalar products
\begin{equation}\label{eq:scalar_prod_y_j_dual_cone}
\lan y_j, y_k \ran
=
\frac{\kappa-s}{s^2} + \frac{\delta_{jk}}{\tau_j^2},
\qquad
j, k \in \{0,\ldots, d\},
\end{equation}
where $\delta_{jk}$ denotes the Kronecker delta.

The $d$-dimensional faces of the cone $\pos(x_0, \ldots, x_d)$ have the form $\pos(x_0, \ldots, x_{j-1}, x_{j+1}, \ldots, x_d)$, for $j\in \{0,\ldots, d\}$. By conic duality, each such $d$-dimensional face corresponds to a $1$-dimensional face of the dual cone, which is a ray spanned by some vector $\widetilde{x}_j$ that is orthogonal to all $x_k$ with $k \neq j$ and satisfies $\lan \widetilde{x}_j, x_j \ran =-1$. Thus, the dual cone of $\pos(x_0, \ldots, x_d)$ is the positive hull of vectors $\widetilde{x}_0, \ldots, \widetilde{x}_d$ with the property that
$$
\lan x_j, \widetilde{x}_k \ran
=
\begin{cases}
0, \quad &j\neq k,\\
-1, \quad &j=k,
\end{cases}
\qquad j,k \in \{0, \ldots, d\}.
$$
Let $A$ be the matrix with $x_0, \ldots, x_d$ as its columns, and let $B$ be the matrix with columns $\widetilde{x}_0, \ldots, \widetilde{x}_d$. The above relations mean that $-B^{\top}A$ is the $(d+1) \times (d+1)$ identity matrix.
The Gram matrix of $\widetilde{x}_0, \ldots, \widetilde{x}_d$, whose entries are $\lan \widetilde{x}_j, \widetilde{x}_k \ran$, is hence given by
$$
B^{\top}B
=
(-A^{-1})(-A^{-1})^{\top}
=
(A^{\top}A)^{-1}.
$$
Using~\eqref{eq:scalar_prod_orthocentric_vectors}, the entries of $A^\top A$ are
\begin{align*}
	\left\lan x_j, x_k \right\ran
	=
	\frac{1}{\kappa} + \lan v_j, v_k \ran
	=
	\frac{1}{\kappa} - \frac{1}{s} + \frac{\delta_{jk}}{\tau_j^2}
	=
	\frac{s - \kappa}{\kappa s} + \frac{\delta_{jk}}{\tau_j^2},
	\qquad \text{ for all } j,k \in \{0, \ldots, d\}.
\end{align*}
 The inverse of the matrix $A^\top A$ has the entries
$$
\lan \widetilde{x}_j, \widetilde{x}_k \ran
=
\frac{\kappa-s}{s^2} \tau_j^2 \tau_k^2  + \delta_{jk} \tau_j^2,
\qquad \text{ for all } j,k \in \{0, \ldots, d\}.
$$
This can be seen by verifying that the product of these matrices is the identity matrix.
Setting $y_j := \widetilde{x}_j/\tau_j^2 \in \R^{d+1}$ for all $j\in \{0, \ldots, d\}$ yields that the dual cone of $\pos (x_0,\ldots, x_d)$ is $\pos (y_0,\ldots, y_d)$, where  $y_0, \ldots, y_d$ satisfy~\eqref{eq:scalar_prod_y_j_dual_cone}.

By the double dual theorem, the dual cone of $\pos (y_0,\ldots, y_d)$ is $\pos (x_0,\ldots, x_d)$. From the definition of the dual cone it  follows that
\begin{align*}
\P\left[N \in \pos(x_0, \ldots, x_d) \right]
&=
\P\left[\lan N,y\ran \leq 0 \quad \text{ for all } y\in \pos(y_0, \ldots, y_d) \right]\\
&=
\P\left[\lan N, y_0 \ran \leq 0,\ldots, \lan N, y_d \ran\leq 0 \right],
\end{align*}
where we recall that $N$ denotes a $(d+1)$-dimensional multivariate standard Gaussian distributed random vector.
The random vector $(\lan N, y_0 \ran, \ldots, \lan N, y_d \ran)^{\top}$, being a linear transformation of $N$,  is again multivariate Gaussian distributed, with mean $0$ and covariance matrix
$$
(\Cov(\lan N, y_j \ran , \lan N, y_k \ran))_{j,k=0}^d
=
(\lan y_j,y_k\rangle)_{j,k=0}^d
=
\Bigl(\frac{\kappa-s}{s^2} + \frac{\delta_{jk}}{\tau_j^2}\Bigr)_{j,k=0}^d.
$$

Let now $\xi, \xi_0, \ldots, \xi_d$ be independent standard normal random variables. Then, we have the following distributional equality of random vectors:
$$
(\lan N, y_0 \ran, \ldots, \lan N, y_d \ran)^\top
\eqdistr
\left( \frac{\xi_0}{\tau_0}-\sqrt{\frac{\kappa-s}{s^2}} \xi, \ldots, \frac{\xi_d}{\tau_d}-\sqrt{\frac{\kappa-s}{s^2}} \xi \right)^\top.
$$
Indeed, both sides are multivariate Gaussian distributed, and their expectations and covariance matrices coincide. Observe that this is the place where we used the assumption $\kappa \geq s$ -- otherwise the square root on the right-hand side is not real. The distributional equality yields
\begin{align*}
\P\left[N \in \pos(x_0, \ldots, x_d) \right]
=&
\P\left[ \lan N, y_0 \ran \leq 0,\ldots, \lan N, y_d \ran\leq 0\right]
\\
=&
\P\left[\frac{\xi_0}{\tau_0}-\sqrt{\frac{\kappa-s}{s^2}} \xi \leq 0, \ldots, \frac{\xi_d}{\tau_d}-\sqrt{\frac{\kappa-s}{s^2}} \xi \leq 0  \right]\\
=&
\P\left[ \xi_0 \leq \frac{\tau_0}s \sqrt{\kappa-s} \, \xi, \ldots, \xi_d \leq \frac{\tau_d}s \sqrt{\kappa-s}\, \xi \right]\\
=&
\frac 1{\sqrt{2\pi}} \int_{-\infty}^{+\infty} \P\left[ \xi_0 \leq \frac{\tau_0}s \sqrt{\kappa-s}  \, x, \ldots, \xi_d \leq \frac{\tau_d}s \sqrt{\kappa-s} \, x \right] \eee^{-x^2/2} \dd x\\
=&
\frac 1{\sqrt{2\pi}} \int_{-\infty}^{+\infty} \prod_{j=0}^d \Phi \left(\frac{\tau_j}s \sqrt{\kappa-s} \,  x \right) \eee^{-x^2/2} \dd x.
\end{align*}
Here, we conditioned on $\{\xi = x\}$ and used independence of $\xi, \xi_0, \ldots, \xi_d$. The claim now follows by splitting the integral on the right-hand side into  integrals over $(-\infty,0]$ and $[0,\infty)$ and applying the substitution $x\mapsto -x$ in the former integral.
\end{proof}

\subsection{Complex asymptotics of the standard normal distribution function}
In this section we recall an asymptotic result for $\Phi(z)$ that will be essential in the sequel. The next lemma can be found in~\cite[Equation~7.12.1]{NIST:DLMF}.  Let $\arg z$ denote the principal value of the argument taking values in $(-\pi,\pi]$.
\begin{lemma}\label{lemma:Phy_asymp}
Fix some $\eps > 0$. The following asymptotics hold as $|z| \to \infty$ provided $z$ stays in the specified sector:
\begin{equation}\label{eq:Phi_asymp}
\Phi(z) =
\begin{cases}
-\frac{1 + o(1)}{\sqrt{2\pi} z} \, \eee^{-z^2/2}, & \text{if } |\arg z| \geq \frac{\pi}{4} + \eps, \\
1 - \frac{1 + o(1)}{\sqrt{2\pi} z} \, \eee^{-z^2/2}, & \text{if } |\arg z| \leq \frac{3\pi}{4} - \eps.
\end{cases}
\end{equation}
\end{lemma}
A proof, in the more general setting of the incomplete Gamma function $\Gamma(a,z) := \int_z^{\infty} t^{a-1} \eee^{-t} \dd t$ which is related to $\Phi$ via $\Phi(z) = 1 - \Gamma(\frac 12, \frac {z^2}{2})/(2\sqrt \pi)$, can be found in~\cite[pp.~109--112]{olver_book_asymptotics}.
A proof in the generality stated here can be found in~\cite[Lemma~3.10]{kabluchkoklimovsky}.
Note that the sets $|\arg z| \geq \frac{\pi}{4} + \eps$ and $|\arg z| \leq \frac{3\pi}{4} - \eps$ have a non-trivial overlap consisting of two sectors. On this intersection, both cases of~\eqref{eq:Phi_asymp} apply.

\begin{corollary}\label{cor:Phi_convergence_2}
We have
$$
\Phi(z) \to
\begin{cases}
1, & \text{as } |z| \to \infty \text{ while staying in the sector } |\arg z| \leq \frac{\pi}{4}, \\
0, & \text{as } |z| \to \infty \text{ while staying in the sector } |\arg (-z)| \leq \frac{\pi}{4}.
\end{cases}
$$
\end{corollary}
\begin{proof}
We consider only the case when $|\arg z| \leq \tfrac \pi4$ since the other case can be reduced to this one by the identity $\Phi(-z) = 1-\Phi(z)$. For $|\arg z| \leq \tfrac \pi4$, the second case of~\eqref{eq:Phi_asymp} gives
$$
\big| 1- \Phi(z) \big|
=
\frac{1+o(1)}{\sqrt{2\pi} \, |z|} \big| \eee^{-z^2/2} \big|
=
\frac{1+o(1)}{\sqrt{2\pi} \, |z|} \eee^{-\Re(z^2)/2},
\qquad \text{ as } |z|\to +\infty.
$$
Next observe that $|\arg z| \leq \tfrac \pi4$ implies  $\Re(z^2) \geq 0$ and hence $\eee^{-\Re(z^2)/2}\leq 1$. Due to the presence of the term $1/|z|$, it follows that $1-\Phi(z) \to 0$.
\end{proof}

\begin{corollary}\label{cor:Phi_convergence_3}
The function $\Phi(z)$ stays bounded on the sectors $\{z\in \CC:  |\arg z| \leq \frac \pi 4\}$ and $\{z\in \CC:  |\arg (-z)| \leq \frac \pi 4\}$.
\end{corollary}
\begin{proof}
Follows from the continuity of $\Phi(z)$ combined with Corollary~\ref{cor:Phi_convergence_2}.
\end{proof}

\subsection{Analytic continuation of an integral}\label{subsec:analytic_continuation_integral_Phi}

In this section, we study improper integrals of the form $\int_{0}^{+\infty} \prod_{j=0}^d \Phi (\lambda_j \omega y )\cdot  \eee^{-(\omega y)^2/2} \omega \dd y$. We begin with a lemma proved by partial integration.

\begin{lemma}\label{lemma:integral_identity_for_angle_analytic}
Let $d \in \N$ and  $\lambda_0, \ldots, \lambda_d \in \CC\backslash\{\pm \ii\}$. Then, for every  $0 < A < B < \infty$ and $\omega \in \CC \backslash\{0\}$ we have
\begin{align}
\int_{A}^{B} \prod_{j=0}^d \Phi\left(\lambda_j \omega y \right)\, \cdot\,  & \eee^{-(\omega y)^2/2} \omega \dd y   \label{eq:integral_identity_for_angle_analytic} \\
\notag
=&
-\frac {1}{\omega} \Bigg( \prod_{j=0}^d \Phi\left(\lambda_j \omega \sqrt{x} \right) \cdot \frac{1}{\sqrt x} \eee^{-\omega^2x/2} \Bigg) \Bigg|_{A^2}^{B^2}
\\
\notag
&-
\frac 1{2\omega} \int_{A^2}^{B^2} \prod_{j=0}^d \Phi\left(\lambda_j\omega \sqrt{x} \right) \cdot \frac {1}{x^{3/2}} \eee^{-\omega^2 x/2}  \dd x\\ \notag
&-
\sum_{\ell=0}^d \frac{\lambda_\ell}{\sqrt{2\pi} \omega^2  (1+\lambda_\ell^2)}
\Bigg(\prod_{\substack{j=0 \\ j\neq \ell}}^d \Phi\left(\lambda_j \omega \sqrt{x} \right) \cdot \frac{1}{x} \eee^{-\frac{\omega^2 x}{2}\left(1+\lambda_\ell^2\right)}  \Bigg) \Bigg |_{A^2}^{B^2}\\ \notag
&-\sum_{\ell=0}^d \frac{\lambda_\ell}{\sqrt{2\pi} \omega^2 (1+\lambda_\ell^2)} \int_{A^2}^{B^2} \prod_{\substack{j=0 \\ j\neq \ell}}^d \Phi\left(\lambda_j \omega \sqrt{x} \right) \cdot \frac{1}{x^2} \eee^{-\frac{\omega^2 x}{2}(1+\lambda_\ell^2)} \dd x\\ \notag
&+ \sum_{0\leq \ell_1 \neq \ell_2 \leq d} \frac{\lambda_{\ell_1} \lambda_{\ell_2}}{4\pi \omega (1+ \lambda_{\ell_1}^2)} \int_{A^2}^{B^2} \prod_{\substack{j=0 \\ j\notin \{ \ell_1, \ell_2 \}}}^d \Phi\left(\lambda_j \omega \sqrt{x} \right) \cdot \frac{1}{x^{3/2}} \eee^{-\frac{\omega^2 x}{2}\left(1+ \lambda_{\ell_1}^2 + \lambda_{\ell_2}^2\right)} \dd x.
\end{align}
\end{lemma}

\begin{proof}
Applying the substitution $y^2 = x$ to the left-hand side of~\eqref{eq:integral_identity_for_angle_analytic}, followed by integration by parts, yields
\begin{align*}
\int_A^B \prod_{j=0}^d \Phi\left(\lambda_j\omega y \right) \,\cdot\, & \eee^{-(\omega y)^2/2} \omega \dd y
\\
=&
\int_{A^2}^{B^2} \prod_{j=0}^d \Phi\left(\lambda_j \omega \sqrt x \right) \cdot \eee^{-\omega^2 x/2}  \frac{\omega \dd x}{2 \sqrt x}\\
=&
-\frac 1 \omega \int_{A^2}^{B^2} \prod_{j=0}^d \Phi\left(\lambda_j \omega \sqrt x \right) \cdot \frac{1}{\sqrt x} \cdot  \dd \eee^{-\omega^2 x/2} \\
=&
-\frac 1{\omega} \Bigg( \prod_{j=0}^d \Phi\left(\lambda_j \omega \sqrt x \right) \cdot \frac{1}{\sqrt x} \eee^{-\omega^2 x/2} \Bigg) \Bigg|_{A^2}^{B^2}
\\
&+ \frac 1{\omega} \cdot \left(-\frac{1}{2}\right) \int_{A^2}^{B^2} \prod_{j=0}^d \Phi\left(\lambda_j \omega \sqrt x \right) \cdot \frac{1}{x^{3/2}} \eee^{-\omega^2 x/2} \dd x\\
&+
\frac 1{\omega} \sum_{\ell = 0}^d \int_{A^2}^{B^2} \prod_{\substack{j=0 \\ j\neq \ell}}^d \Phi\left(\lambda_j \omega \sqrt x \right) \cdot \frac 1{\sqrt{2 \pi}} \eee^{-\frac{\omega^2 x}{2} \lambda_\ell^2} \lambda_\ell
\omega \cdot  \frac 1{2\sqrt x} \cdot  \frac{1}{\sqrt x}  \cdot  \eee^{-\omega^2 x/2} \dd x.
\end{align*}
The first two summands on the right-hand side agree with those on the right-hand side of~\eqref{eq:integral_identity_for_angle_analytic}. The final term can be rewritten to
\begin{align*}
\sum_{\ell = 0}^d \frac{\lambda_\ell}{2 \sqrt{2 \pi}} \int_{A^2}^{B^2} \prod_{\substack{j=0 \\ j\neq \ell}}^d \Phi\left(\lambda_j \omega \sqrt x \right) \cdot & \frac{1}{x} \eee^{-\frac {\omega^2 x}2 \left(1+ \lambda_\ell^2\right)} \dd x\\
=&
- \sum_{\ell=0}^d \frac{\lambda_{\ell}}{\sqrt{2\pi} \omega^2  (1+ \lambda_\ell^2 )} \int_{A^2}^{B^2} \prod_{\substack{j=0 \\ j\neq \ell}}^d \Phi\left(\lambda_j \omega \sqrt x \right)  \cdot \frac 1x \cdot   \dd \eee^{-\frac {\omega^2 x}2 \left(1+\lambda_\ell^2\right)}.
\end{align*}
Note that $1+\lambda_\ell^2 \neq 0$ due to the assumption $\lambda_\ell \neq \pm \ii$. Integration by parts shows that this is equal to the sum of the final three terms on the right-hand side of~\eqref{eq:integral_identity_for_angle_analytic}.
\end{proof}

\begin{proposition}\label{prop:integral_on_upper_halfplane}
Fix $d\in \N$ and $\mu_0, \ldots, \mu_d \in \R\backslash\{0\}$. Consider the set
$$
S_+:= \{z\in \CC: \Im z \geq 0\} \backslash \{-\mu_0^{-2},\ldots, -\mu_d^{-2}\}.
$$
\begin{itemize}
\item[(a)] Let $\omega = 1-\ii$. Then, for every  $z \in S_+$ and  every $A>0$, we have
\begin{align}
\int_0^{\omega \infty} & \prod_{j=0}^d   \Phi\left(\mu_j \sqrt{z}  x \right) \cdot \eee^{-\frac12 {x^2}}  \dd x
\eqdef
\lim_{B \to +\infty} \int_0^{B} \prod_{j=0}^d \Phi\left(\mu_j \sqrt{z}  \omega y \right) \cdot \eee^{-\frac12 {\omega^2y^2}} \omega \dd y
\label{eq:anal_cont_int_improper_int}\\
=&
\int_0^{A} \prod_{j=0}^d \Phi\left(\mu_j \sqrt{z}  \omega y \right) \cdot \eee^{-\frac12 {\omega^2y^2}} \omega \dd y
\label{eq:anal_cont_int_term_0}
\\
&+\frac {1}{A \omega}  \prod_{j=0}^d \Phi\left(\mu_j \sqrt z \omega A \right) \cdot \eee^{-\omega^2 A^2/2}
\label{eq:anal_cont_int_term_1}
\\
&-
\frac 1{2\omega} \int_{A^2}^{+\infty} \prod_{j=0}^d \Phi\left(\mu_j \sqrt z\omega \sqrt{x} \right) \cdot \frac {1}{x^{3/2}} \eee^{-\omega^2 x/2}  \dd x
\label{eq:anal_cont_int_term_2}\\
&+
\sum_{\ell=0}^d \frac{\mu_\ell \sqrt z}{\sqrt{2\pi} \omega^2  (1+\mu_\ell^2 z)}
\prod_{\substack{j=0 \\ j\neq \ell}}^d \Phi\left(\mu_j  \sqrt z \omega A \right) \cdot \frac{1}{A^2} \eee^{-\frac{\omega^2 A^2}{2}\left(1+\mu_\ell^2 z\right)}
\label{eq:anal_cont_int_term_3}
\\
&-\sum_{\ell=0}^d \frac{\mu_\ell \sqrt z}{\sqrt{2\pi} \omega^2 (1+\mu_\ell^2 z)} \int_{A^2}^{+\infty} \prod_{\substack{j=0 \\ j\neq \ell}}^d \Phi\left(\mu_j \sqrt z  \omega \sqrt{x} \right) \cdot \frac{1}{x^2} \eee^{-\frac{\omega^2 x}{2}(1+\mu_\ell^2z)} \dd x
\label{eq:anal_cont_int_term_4}
\\
&+ \sum_{0\leq \ell_1 \neq \ell_2 \leq d} \frac{\mu_{\ell_1} \mu_{\ell_2} z}{4\pi \omega (1+ \mu_{\ell_1}^2 z)} \int_{A^2}^{+\infty} \prod_{\substack{j=0 \\ j\notin \{ \ell_1, \ell_2 \}}}^d \Phi\left(\mu_j \sqrt z  \omega \sqrt{x} \right) \cdot \frac{1}{x^{3/2}} \eee^{-\frac{\omega^2 x}{2}\left(1+ \mu_{\ell_1}^2 z + \mu_{\ell_2}^2 z\right)} \dd x,
\label{eq:anal_cont_int_term_5}
\end{align}
where the square root is defined by the convention $\sqrt{r \eee^{\ii\theta}} = \sqrt r \eee^{\ii \theta/2}$ for $r\geq 0$ and $\theta \in [0,\pi]$.
In particular, the limit in~\eqref{eq:anal_cont_int_improper_int} exists and is finite. All the integrals in~\eqref{eq:anal_cont_int_term_2}, \eqref{eq:anal_cont_int_term_4}, \eqref{eq:anal_cont_int_term_5}  converge absolutely.
\item[(b)] Equations~\eqref{eq:anal_cont_int_improper_int}--\eqref{eq:anal_cont_int_term_5}  define a function of  $z$ which is continuous on the set $S_+$ and analytic on the open upper half-plane $\{z\in \CC: \Im z >0\}$.
\end{itemize}
\end{proposition}
\begin{proof}
\emph{Proof of (a).}
Split the integral $\int_0^B$ into the two parts,  $\int_0^A$ and $\int_A^B$. The first integral exists since the integrand is continuous on $[0,A]$, and gives~\eqref{eq:anal_cont_int_term_0}.   For the second integral, we  apply Lemma~\ref{lemma:integral_identity_for_angle_analytic} with $\omega = 1-\ii$ and $\lambda_j = \mu_j \sqrt z$:
\begin{align}
\int_{A}^{B} & \prod_{j=0}^d \Phi\left(\mu_j  \sqrt z\omega y \right)\, \cdot\,   \eee^{-(\omega y)^2/2} \omega \dd y   \label{eq:integral_identity_for_angle_analytic_repeat} \\
=&
-\frac {1}{\omega} \Bigg( \prod_{j=0}^d \Phi\left(\mu_j \sqrt z \omega \sqrt{x} \right) \cdot \frac{1}{\sqrt x} \eee^{-\omega^2x/2} \Bigg) \Bigg|_{x=A^2}^{x=B^2}
\label{eq:anal_cont_int_term_1_proof}
\\
&-
\frac 1{2\omega} \int_{A^2}^{B^2} \prod_{j=0}^d \Phi\left(\mu_j \sqrt z\omega \sqrt{x} \right) \cdot \frac {1}{x^{3/2}} \eee^{-\omega^2 x/2}  \dd x
\label{eq:anal_cont_int_term_2_proof}
\\
&-
\sum_{\ell=0}^d \frac{\mu_\ell \sqrt z}{\sqrt{2\pi} \omega^2  (1+\mu_\ell^2 z)}
\Bigg(\prod_{\substack{j=0 \\ j\neq \ell}}^d \Phi\left(\mu_j  \sqrt z \omega \sqrt{x} \right) \cdot \frac{1}{x} \eee^{-\frac{\omega^2 x}{2}\left(1+\mu_\ell^2 z\right)}  \Bigg) \Bigg |_{x=A^2}^{x=B^2}
\label{eq:anal_cont_int_term_3_proof}
\\
&-\sum_{\ell=0}^d \frac{\mu_\ell \sqrt z}{\sqrt{2\pi} \omega^2 (1+\mu_\ell^2 z)} \int_{A^2}^{B^2} \prod_{\substack{j=0 \\ j\neq \ell}}^d \Phi\left(\mu_j \sqrt z  \omega \sqrt{x} \right) \cdot \frac{1}{x^2} \eee^{-\frac{\omega^2 x}{2}(1+\mu_\ell^2z)} \dd x
\label{eq:anal_cont_int_term_4_proof}
\\
&+ \sum_{0\leq \ell_1 \neq \ell_2 \leq d} \frac{\mu_{\ell_1} \mu_{\ell_2} z}{4\pi \omega (1+ \mu_{\ell_1}^2 z)} \int_{A^2}^{B^2} \prod_{\substack{j=0 \\ j\notin \{ \ell_1, \ell_2 \}}}^d \Phi\left(\mu_j \sqrt z  \omega \sqrt{x} \right) \cdot \frac{1}{x^{3/2}} \eee^{-\frac{\omega^2 x}{2}\left(1+ \mu_{\ell_1}^2 z + \mu_{\ell_2}^2 z\right)} \dd x.
\label{eq:anal_cont_int_term_5_proof}
\end{align}
The assumption $z \neq  -\mu_j^{-2}$, $j=0,\ldots, d$, entails that all denominators appearing on the right-hand side are non-zero.

Let us analyze the terms on the right-hand side. First of all, observe that $\omega^2 = (1-\ii)^2 = -2\ii$. It follows that for all $x\geq A$ and $z\in \CC$ such that $\Im z \geq 0$, and for all possible choices of indices $\ell, \ell_1,\ell_2$,  we have
$$
|\eee^{-\omega^2x/2}| = 1,
\qquad
|\eee^{-\frac{\omega^2 x}{2}(1+\mu_\ell^2z)}|\leq 1
\qquad
\left|\eee^{-\frac{\omega^2 x}{2}\left(1+ \mu_{\ell_1}^2 z + \mu_{\ell_2}^2 z\right)}\right|\leq 1.
$$
Next, observe that for $z\in \CC$ satisfying $\Im z \geq 0$ we have $\sqrt z \omega \sqrt{x} \in T$, where $T:=\{u\in \CC:  |\arg u| \leq \frac \pi 4\}$.
Recall from Corollary~\ref{cor:Phi_convergence_3} that  $|\Phi(u)|\leq C$ and $\Phi(-u) \leq C$ for all  $u\in T$ and some absolute constant $C>0$. It follows that, regardless of the sign of $\mu_j  \in \R\backslash\{0\}$, we have
$$
\Phi\left(\mu_j \sqrt z \omega \sqrt{x} \right) \leq C,
$$
for all $z\in \CC$ such that $\Im z\geq 0$, all  $x>0$  and $j\in \{0, \ldots, d\}$.

To summarize: all terms of the form $\Phi(\ldots)$ or $\exp\{\ldots\}$ that appear in~\eqref{eq:anal_cont_int_term_1_proof}--\eqref{eq:anal_cont_int_term_5_proof} are uniformly bounded.  With this information at hand, we can let $B\to+\infty$ and  consider each summand individually. The terms in~\eqref{eq:anal_cont_int_term_1_proof} and~\eqref{eq:anal_cont_int_term_3_proof} vanish at $x=B^2 \to+\infty$ due to the presence of the factors $1/\sqrt x$ and $1/x$ that converge to $0$ and since all remaining factors are bounded. Evaluating these terms at $x=A^2$ gives~\eqref{eq:anal_cont_int_term_1} and~\eqref{eq:anal_cont_int_term_3}. Next, the integrals of the form $\int_{A^2}^{B^2}$ appearing in~\eqref{eq:anal_cont_int_term_2_proof}, \eqref{eq:anal_cont_int_term_4_proof}, \eqref{eq:anal_cont_int_term_5_proof} converge to the respective integrals of the form $\int_{A^2}^{+\infty}$ due to the presence of the terms $1/ x^{3/2}$ and $1/x^2$ that are absolutely integrable over $[A, \infty)$ and since all remaining terms are uniformly bounded. This yields~\eqref{eq:anal_cont_int_term_2}, \eqref{eq:anal_cont_int_term_4}, and~\eqref{eq:anal_cont_int_term_5}.

\vspace*{2mm}
\noindent
\emph{Proof of (b).}
Since $\Phi(z)$ is analytic, it is clear that~\eqref{eq:anal_cont_int_term_1} and~\eqref{eq:anal_cont_int_term_3}, considered as functions of $z$,  are continuous on $S_+$ and analytic on $\Im z >0$. Next let us consider~\eqref{eq:anal_cont_int_term_2},  since  the remaining summands, \eqref{eq:anal_cont_int_term_0}, \eqref{eq:anal_cont_int_term_4} and~\eqref{eq:anal_cont_int_term_5},  can be analyzed analogously.
To prove continuity of~\eqref{eq:anal_cont_int_term_2} on $S_+$, let $\{z_n\}_{n\in \N} \subseteq S_+$ be a sequence converging to $z\in S_+$. We need to show that
$$
\lim_{n\to\infty}
\int_{A^2}^{+\infty} G(z_n,x) \dd x
=
\int_{A^2}^{+\infty} G(z,x) \dd x,
$$
where $G(z,x) := \prod_{j=0}^d \Phi(\mu_j \sqrt{z} \omega \sqrt{x}) \cdot x^{-3/2} \eee^{-\omega^2 x/2}$. Since for each fixed $x\geq A^2$ we have $G(z_n, x) \to G(z,x)$ as $n\to\infty$,  the claim follows from the Lebesgue dominated convergence theorem upon observing that $|G(z_n,x)| \leq C^{d+1} x^{-3/2}$, which is integrable in $x$.  The analyticity of~\eqref{eq:anal_cont_int_term_2} on $\Im z >0$ follows from the standard theorem on differentiation under the integral sign;
see Lemma~1.1 on p.~409 in~\cite[Chapter~XV, \S 1]{lang_book_complex_analysis}. Indeed, for every fixed $x\geq A^2$, the function $z\mapsto G(z,x)$ is analytic on $\Im z>0$, and the existence of the integrable majorant $|G(z,x)| \leq C^{d+1} x^{-3/2}$ ensures the applicability of the above mentioned Lemma~1.1. The integrals appearing in~\eqref{eq:anal_cont_int_term_4} and~\eqref{eq:anal_cont_int_term_5} can be analyzed analogously, the integrable majorants being $C^d x^{-2}$ and $C^{d-1} x^{-3/2}$. In~\eqref{eq:anal_cont_int_term_0}, the integrand is bounded.
\end{proof}

Proposition~\ref{prop:integral_on_upper_halfplane} concerns an integral over an interval for $\kappa$ in the upper half-plane. The following result is the corresponding analogue for the lower half-plane.

\begin{proposition}\label{prop:integral_on_lower_halfplane}
Fix $d\in \N$ and $\mu_0, \ldots, \mu_d \in \R\backslash\{0\}$. Consider the set
$$
S_-:= \{z\in \CC: \Im z \leq 0\} \backslash \{-\mu_0^{-2},\ldots, -\mu_d^{-2}\}.
$$
\begin{itemize}
\item [(a)] Let now $\omega = 1 + \ii$. Then, for every  $z \in S_-$ and  every $A>0$, Equations~\eqref{eq:anal_cont_int_improper_int}--\eqref{eq:anal_cont_int_term_5} hold with the square root defined by the convention  $\sqrt{r \eee^{\ii\theta}} = \sqrt r \eee^{\ii \theta/2}$ for $r\geq 0$ and $\theta \in [-\pi,0]$ (so that this time $\sqrt {-1} = -\ii$). The limit in~\eqref{eq:anal_cont_int_improper_int} exists and is finite. All the integrals in~\eqref{eq:anal_cont_int_term_2}, \eqref{eq:anal_cont_int_term_4}, \eqref{eq:anal_cont_int_term_5}  converge absolutely.
\item[(b)] Equations~\eqref{eq:anal_cont_int_improper_int}--\eqref{eq:anal_cont_int_term_5}  define a function of  $z$ which is continuous on the set $S_-$ and analytic on the open lower half-plane $\{z\in \CC: \Im z < 0\}$.
\end{itemize}
\end{proposition}
\begin{proof}
This follows from Proposition~\ref{prop:integral_on_upper_halfplane} by taking the complex conjugate. More precisely, let $z\in S_-$, apply Proposition~\ref{prop:integral_on_upper_halfplane} ~(a) to $\bar z\in S_+$, and then take the complex conjugate of both sides of~\eqref{eq:anal_cont_int_improper_int} -- \eqref{eq:anal_cont_int_term_5}. For Part~(b), recall in addition that if $f(z)$ is analytic, then so is $\overline {f(\bar z)}$.
\end{proof}

The next lemma shows that when $\sqrt{z}$ is real, the contour of integration in the integral can be rotated without changing its value.
\begin{proposition}[Rotating the ray of integration]\label{prop:rotate_contour_integration}
Let $d\in \N$ and $\mu_0,\ldots, \mu_d \in \R\backslash\{0\}$. Then, for every real $z>0$, the value of the integral
\begin{equation}\label{eq:prop:rotate_contour_integration}
\int_0^{\omega \infty} \prod_{j=0}^d \Phi\left(\mu_j \sqrt{z} x \right) \cdot \eee^{-x^2/2} \dd x
\eqdef
\lim_{B \to +\infty} \int_0^{B} \prod_{j=0}^d \Phi\left(\mu_j \sqrt{z}  \omega y \right) \cdot \eee^{-\omega^2 y^2/2}\omega  \dd y
\end{equation}
does not depend on the choice of $\omega \in \CC \backslash \{0\}$ with $|\arg \omega| \leq \pi/4$.
\end{proposition}

\begin{proof}
	We first prove that the limit in~\eqref{eq:prop:rotate_contour_integration} exists for all  $\omega \in \CC \backslash \{0\}$ with $|\arg \omega| \leq \tfrac \pi4$. For $\omega = 1 - \ii$, this follows from Proposition~\ref{prop:integral_on_upper_halfplane}\,(a), and for $\omega = 1 + \ii$, it follows from Proposition~\ref{prop:integral_on_lower_halfplane}\,(a). Since the limit in~\eqref{eq:prop:rotate_contour_integration} does not change if we replace $\omega$ by $\lambda \omega$ with any $\lambda >0$, this settles the cases $\arg \omega = \pm \frac \pi 4$.
	Now, let $- \tfrac\pi4 < \arg \omega < \tfrac \pi4$. In this case, we prove a stronger claim, namely
	$$
	\lim_{B \to +\infty} \int_0^{B} \left|\prod_{j=0}^d \Phi\left(\mu_j \sqrt{z}  \omega y \right) \cdot \eee^{-\omega^2 y^2/2}\omega \right| \dd y < \infty.
	$$
	The term
	$
	\Phi(\mu_j \sqrt{z} \omega y )
	$
	is bounded  for $y\in [0,+\infty)$ by Corollary~\ref{cor:Phi_convergence_3}. Moreover,
	$
	| \eee^{-\omega^2y^{2}/2} |
	=
	\eee^{-\Re(\omega^{2})\, y^{2} / 2}
	$
	for $y\in [0,+\infty)$.
	Since $\Re(\omega^{2}) > 0$, the exponent is negative, hence $|\eee^{-\omega^2y^{2}/2}|$ tends to $0$ exponentially fast as $y \to +\infty$. Together, this proves the absolute convergence of the integral.
	
	Next, we show that the value of the integral does not depend on the chosen $\omega$.
	
	Let $\omega \in \CC \backslash \{0\}$ with $\arg \omega \in [-\tfrac{\pi}4, \tfrac{\pi}4] \backslash \{0\}$ and denote $\delta = \sgn (\arg \omega)$.
	Considering the circular arc $\gamma_{A, \omega}: [0,|\arg(\omega)|] \to \CC, \theta \mapsto A \eee^{\ii\, \delta \theta}$, by Cauchy's theorem it suffices to show that
	$$
	\lim_{A \rightarrow \infty}
	\int_{\gamma_{A, \omega}} \Big| \prod_{j=0}^d \Phi\left( \mu_j \sqrt{z} x \right) \cdot \eee^{-x^2/2} \dd x \Big|
	= 0.
	$$
	By Corollary~\ref{cor:Phi_convergence_3},  $\Phi(\mu_j \sqrt{z} x)$ is bounded on $\{x \in \CC: |\arg x| \leq \frac{\pi}4 \}$. Hence, it suffices to show that
	$
	\int_0^{\pi/4} |\eee^{-(A\eee^{\ii \, \delta \theta})^2/2} A\, \ii \, \delta \eee^{\ii\, \delta \theta} | \dd \theta \to 0
	$
	as $A \to +\infty$.
	For all $\theta$ in the interval of integration, we have
$$
	\Big| \eee^{-(A\eee^{\ii\,\delta \theta})^2/2} A\, \ii \, \delta \eee^{\ii \, \delta \theta} \Big|
	=
	A\eee^{\Re(-(A\eee^{\ii\, \delta \theta})^2/2)}
	=
	A \eee^{-\frac{A^2}{2} \Re((\eee^{\ii\,\delta \theta})^2)}
	=
	A \eee^{-\frac{A^2}{2} \cos(2 \theta)}
	=
	A \eee^{-\frac{A^2}{2} \sin(\frac{\pi}{2}-2\theta)}.
$$
	The substitution $x=\tfrac{\pi}4-\theta$ yields
	$$
	\int_0^{\pi/4} |\eee^{-(A \eee^{\ii\, \delta \theta})^2/2} A\, \ii \, \delta \eee^{\ii\, \delta \theta}| \dd \theta
	=
	\int_0^{\pi/4} A \eee^{-\frac{A^2}{2} \sin(2x)} \dd x.
	$$
There exists some $c>0$ such that $\sin (2x) \geq cx$ for all $x\in [0,\pi/4]$, and we can conclude
	$$
	\int_0^{\pi/4} |\eee^{-(A\eee^{\ii\, \delta \theta})^2/2} A\, \ii \, \delta \eee^{\ii \,\delta \theta}| \dd \theta
	\leq
	\int_0^{\pi/4} A \eee^{-\frac{A^2}2 c x} \dd x
	=
	A\frac{2}{c A^2} \left( 1-\eee^{-\frac{A^2}2 c \frac{\pi}4} \right)
	\xrightarrow[A \to +\infty]{}
	0.
	$$
\end{proof}

Let us summarize what we have proved so far.
\begin{proposition}\label{prop:summary_properties_anal_cont_int}
Let $d\in \N$ and $\mu_0,\ldots, \mu_d \in \R\backslash\{0\}$.
\begin{itemize}
\item [(a)] The  function
\begin{equation}\label{eq:def_F_+_integral}
F_{+}(z) := \frac 1{\sqrt{2\pi}}\int_{0}^{(1-\ii) \infty} \left( \prod_{j=0}^d \Phi \left(\mu_j\sqrt{z} x \right) + \prod_{j=0}^d \Phi \left(- \mu_j \sqrt{z} x \right) \right) \eee^{-\frac{x^2}2} \dd x
\end{equation}
is well defined  and continuous on the set $S_+= \{z\in \CC: \Im z \geq 0\} \bsl \{-\mu_0^{-2},\ldots, -\mu_d^{-2}\}$, with the convention  $\sqrt{r \eee^{\ii\theta}} = \sqrt r \eee^{\ii \theta/2}$ for $r\geq 0$ and $\theta \in [0,\pi]$.  Moreover, $F_+(z)$ is an analytic function on the open upper half-plane $\{z\in \CC:\Im z >0\}$.
\item [(b)] The  function
\begin{equation}\label{eq:def_F_-_integral}
F_{-}(z) := \frac 1{\sqrt{2\pi}} \int_{0}^{(1+\ii) \infty} \left( \prod_{j=0}^d \Phi \left(\mu_j\sqrt{z} x \right) + \prod_{j=0}^d \Phi \left(- \mu_j \sqrt{z} x \right) \right) \eee^{-\frac{x^2}2} \dd x
\end{equation}
is well defined  and continuous on the set $S_-= \{z\in \CC: \Im z \leq 0\} \bsl \{-\mu_0^{-2},\ldots, -\mu_d^{-2}\}$, with the convention  $\sqrt{r \eee^{\ii\theta}} = \sqrt r \eee^{\ii \theta/2}$ for $r\geq 0$ and $\theta \in [-\pi,0]$.  Moreover, $F_-(z)$ is an analytic function on the open  lower half-plane $\{z\in \CC: \Im z < 0\}$.
\item [(c)] For every real $z>0$, we have
$$
F_{+}(z) = F_{-}(z) = \frac 1{\sqrt{2\pi}}  \int_{0}^{+\infty} \left( \prod_{j=0}^d \Phi \left(\mu_j\sqrt{z} x \right) + \prod_{j=0}^d \Phi \left(- \mu_j \sqrt{z} x \right) \right) \eee^{-\frac{x^2}2} \dd x.
$$
On the right-hand side, we can replace $\int_0^{+\infty}$  by $\int_0^{\omega \infty}$, for every $\omega\in \CC\bsl \{0\}$ with $|\arg \omega|\leq \frac \pi 4$.
\item [(d)] The following function is analytic on $\CC\backslash (-\infty, 0]$:
$$
F(z) :=
\begin{cases}
F_{+}(z) \text{ if } & \Im z >0 \text{ or } z\in (0,\infty),
\\
F_-(z)  \text{ if } & \Im z < 0 \text{ or } z\in (0,\infty).
\end{cases}
$$
\item [(e)]
For every $z\in  S_+$ we have $F_+(z) = \overline {F_-(\bar z)}$.
\end{itemize}
\end{proposition}
\begin{proof}
Parts (a) and (b) follow from Propositions~\ref{prop:integral_on_upper_halfplane} and~\ref{prop:integral_on_lower_halfplane}, respectively. Part~(c) follows from Proposition~\ref{prop:rotate_contour_integration}. Since $F_+$ and $F_-$ agree on $[0,+\infty)$, Part~(d) then follows from the Schwarz reflection principle; see Theorem~1.1 on p.~294 in~\cite[Chapter~IX, \S 1]{lang_book_complex_analysis}. Part~(e) follows from the same principle or directly from the definitions of $F_+$ and $F_-$ given in~\eqref{eq:def_F_+_integral} and~\eqref{eq:def_F_-_integral}.
\end{proof}

\subsection{Proof of Theorem~\ref{theo:vol_acute_orthoc_simpl_in_Klein_model_any_dimension_1-i}}
\label{subsec:proof_formula_hyperbolic_volume_complete}

We are now in position to complete the proof of our main result, Theorem~\ref{theo:vol_acute_orthoc_simpl_in_Klein_model_any_dimension_1-i}, which we restate here for convenience.
\begin{theorem}\label{theo:vol_acute_orthoc_simpl_in_Klein_model_any_dimension_1-i_proof}
Fix $d\geq 2$, $\tau_0,\ldots, \tau_d>0$ and define $s=\tau_0^2 +\ldots + \tau_d^2$. Consider a simplex $Q := [v_0, \ldots, v_d]\subseteq \R^d$, where  $v_0,\ldots, v_d\in \R^d$ are vectors in orthocentric position  with parameters $\tau_0,\ldots, \tau_d$; see Definition~\ref{def:orthocentric_position}. Then, for all $\kappa \neq 0$ such that $Q\subseteq \bar{\BB}^d(\kappa)$, we have
\begin{equation}\label{eq:vol_acute_orthoc_simpl_in_Klein_model_even_and_odd_upper_proof}
\Vol_{d,\kappa}(Q)
=
\frac{\omega_{d+1}}{\sqrt{\kappa}^{d}} \frac1{\sqrt{2\pi}} \int_{0}^{(1-\ii)\infty} \left( \prod_{j=0}^d \Phi \left( \frac{\tau_j}{s} \sqrt{\kappa -s} x \right) + \prod_{j=0}^d \Phi \left(- \frac{\tau_j}{s} \sqrt{\kappa -s} x \right) \right) \eee^{-\frac{x^2}2} \dd x,
\end{equation}
with the convention that $\sqrt{-r} = \ii \sqrt r$ for $r\geq 0$.  The formula remains valid if $\int_{0}^{(1-\ii)\infty}$ is replaced by $\int_{0}^{(1+\ii)\infty}$, this time with the opposite convention  $\sqrt{-r} = -\ii \sqrt r$ for $r\geq 0$.
\end{theorem}
\begin{proof}
Without loss of generality, let $\tau_0 = \min \{\tau_0,\ldots, \tau_d\}$. Recall from  Lemma~\ref{lemma:lower_limit_kappa_simplex} that $Q \subseteq \bar \BB^d(\kappa)$ is equivalent to $\kappa \geq  \kappa_0$, where 
$$
\kappa_0
= -\frac{\tau_0^2 s}{s-\tau_0^2} <0.
$$

\vspace*{2mm}
\noindent
\emph{The case where $\kappa\geq s$.}
For $\kappa\geq s$, we already know from Proposition~\ref{prop:vol_spher_orthocentr_simplex_kappa_greater_s} that
\begin{equation}\label{eq:vol_equals_int_proof_main_theo}
\frac{\Vol_{d,\kappa}(Q)}{\omega_{d+1}}
=
\frac{1}{\sqrt{\kappa}^{d}} \frac1{\sqrt{2\pi}} \int_{0}^{+\infty} \left( \prod_{j=0}^d \Phi \left( \frac{\tau_j}{s} \sqrt{\kappa -s} x \right) + \prod_{j=0}^d \Phi \left(- \frac{\tau_j}{s} \sqrt{\kappa -s} x \right) \right) \eee^{-\frac{x^2}2} \dd x.
\end{equation}
By Proposition~\ref{prop:rotate_contour_integration}, we can replace $\int_{0}^{+\infty}$ by $\int_0^{(1-\ii) \infty}$  or $\int_0^{(1+\ii) \infty}$ without changing the value of the integral. This yields~\eqref{eq:vol_acute_orthoc_simpl_in_Klein_model_even_and_odd_upper_proof} and settles the case where $\kappa \geq s$.

\vspace*{2mm}
\noindent
\emph{Analytic continuation.}
To deal with the case $\kappa \in (\kappa_0, \infty)\backslash\{0\}$, we shall study the analytic continuation of the functions appearing in~\eqref{eq:vol_equals_int_proof_main_theo}.
Consider the function
$$
H(\kappa) :=  \frac{\Vol_{d,\kappa}(Q)}{\omega_{d+1}}  = \frac 1 {\omega_{d+1}} \int_Q \frac{\dd y}{(1+\kappa \| y \|^2)^{(d+1)/2}},
\qquad
\kappa \geq \kappa_0.
$$
Define also the function $F(z)$ as in  Proposition~\ref{prop:summary_properties_anal_cont_int} with $\mu_0 := \tau_0 /s,\ldots, \mu_d := \tau_d /s$. With this notation, \eqref{eq:vol_equals_int_proof_main_theo} takes the form
$$
H(\kappa) = \kappa^{-d/2} F(\kappa-s) \qquad  \text{ for all real } \kappa \geq s.
$$
Now, by Lemma~\ref{lemma:Vol_d_kappa_analytic},  the function $H(\kappa)$  admits an analytic continuation to the domain $\CC\bsl (-\infty, \kappa_0]$. On the other hand,  it follows from Proposition~\ref{prop:summary_properties_anal_cont_int} that the function $\kappa \mapsto \kappa^{-d/2} F(\kappa-s)$ admits analytic continuation to the domain $\CC\backslash (-\infty, s]$. Observe that $\CC\backslash (-\infty, s] \subseteq \CC\bsl (-\infty, \kappa_0]$ since $\kappa_0 < 0 < s$. The uniqueness theorem for analytic functions implies that
\begin{equation}\label{eq:H_equals_F_complex_argument}
H(\kappa) = \kappa^{-d/2} F(\kappa-s) \qquad  \text{ for all } \kappa \in \CC\backslash (-\infty, s].
\end{equation}
From Proposition~\ref{prop:summary_properties_anal_cont_int} (a), (b) we also know that the limits $\lim_{t\downarrow 0} F(z+ \ii t) = F_+(z)$ and $\lim_{t\downarrow 0} F(z - \ii t) = F_-(z)$ exist for all $z\in \R\backslash \{-s^2/\tau_0^2,\ldots, -s^2/\tau_d^2\}$, where we recall that $\mu_j = \tau_j /s$.

Let us verify that $\kappa_0 - s > -s^2/\tau_j^2$ for all $j=0,\ldots, d$ and $d\geq 2$. Indeed,
\begin{align*}
1+\frac{\tau_j^2}{s^2} (\kappa_0-s)
=
1+\frac{\tau_j^2}{s^2} \left(-\frac{\tau_0^2 s}{s-\tau_0^2} -s \right)
=
1-\frac{\tau_j^2s^2}{s^2(s-\tau_0^2)}
=
\frac{s-\tau_0^2-\tau_j^2}{s-\tau_0^2}
> 0,
\end{align*}
where we used that $s= \tau_0^2 + \ldots + \tau_d^2$ and $d\geq 2$ in the last step.

\vspace*{2mm}
\noindent
\emph{Completing the proof.}
Take some $\kappa \in (\kappa_0, \infty)$ with $\kappa \neq 0$. Using first the analyticity of $H$ at $\kappa$, then the equality  $H(z) = z^{-d/2} F_+(z-s)$ that is valid for $\Im z >0$, see~\eqref{eq:H_equals_F_complex_argument},  and finally the continuity of $F_+$ at $\kappa-s > \kappa_0 - s > -s^2/\tau_j^2$, we obtain
$$
H(\kappa) = \lim_{t\downarrow 0} H(\kappa + \ii t) = \lim_{t\downarrow 0} (\kappa+\ii t)^{-d/2}F_+(\kappa + \ii t - s) = {\sqrt \kappa}^{-d} F_+(\kappa-s),
$$
with the convention $\sqrt{\kappa} = \ii |\kappa|^{1/2}$ if $\kappa <0$.  Similarly, approaching $\kappa$ from the lower half-plane, one gets
$$
H(\kappa) = \lim_{t\downarrow 0} H(\kappa - \ii t) = \lim_{t\downarrow 0} (\kappa-\ii t)^{-d/2}F_-(\kappa - \ii t - s) =  {\sqrt \kappa}^{-d} F_-(\kappa-s),
$$
this time with the opposite convention  $\sqrt{\kappa} = -\ii |\kappa|^{1/2}$ if $\kappa <0$.
Taking everything together gives
\begin{align}
\frac{\Vol_{d,\kappa}(Q)}{\omega_{d+1}}  &= H(\kappa) = \frac{F_+(\kappa-s)}{\ii^{d} |\kappa|^{d/2}} = \frac{F_-(\kappa-s)}{(-\ii)^{d} |\kappa|^{d/2}},
\qquad
&\text{ if } \kappa \in (\kappa_0, 0),  \label{eq:Vol_eq_H_eq_F_+_F_-_1}
\\
\frac{\Vol_{d,\kappa}(Q)}{\omega_{d+1}}  &= H(\kappa) = \frac{F_+(\kappa-s)}{\kappa^{d/2}} = \frac{F_-(\kappa-s)}{\kappa^{d/2}},
\qquad
&\text{ if } \kappa \in (0,\infty). \label{eq:Vol_eq_H_eq_F_+_F_-_2}
\end{align}
To complete the proof for $\kappa \in (\kappa_0,\infty)$, recall the definitions of $F_+$ and $F_-$ given in~\eqref{eq:def_F_+_integral}, \eqref{eq:def_F_-_integral}.

Finally, to treat the remaining case $\kappa = \kappa_0$, we first observe that,  by monotone convergence,  $H(\kappa) \uparrow  H(\kappa_0)$ as $\kappa \downarrow \kappa_0$. On the other hand, by Proposition~\ref{prop:summary_properties_anal_cont_int} (a), (b),  the functions $F_+$ and $F_-$ are continuous at $\kappa_0 -s$ since $\kappa_0 - s > -s^2/\tau_j^2$, as we have shown above. Thus, letting $\kappa\downarrow \kappa_0$ in~\eqref{eq:Vol_eq_H_eq_F_+_F_-_1} shows that this equation remains valid for $\kappa=\kappa_0$.
\end{proof}

As a byproduct of the above proof, we obtain a property of the functions $F_+$ and $F_-$ which is  not directly evident from their analytic definition.
\begin{corollary}
Let $d\geq 2$, $\mu_0, \ldots,  \mu_d>0$ and suppose that $\mu_0 = \min \{\mu_0,\ldots, \mu_d\}$.
Put $r = 1/(\mu_1^2 + \ldots + \mu_d^2)$ and $s=1/(\mu_0^2+\ldots + \mu_d^2)$. Then, the following hold.
\begin{itemize}
\item[(a)]  $F_+(-s) = F_-(-s) = 0$.
\item [(b)] If $d$ is even, then $F_+(z) = F_-(z) \in \R$ for all $z\in [-r, \infty)$. Consequently, the function $F$ admits analytic continuation to $\CC \backslash (-\infty, -r]$.
\item[(c)] If $d$ is  odd, then $F_+(z) = - F_-(z) \in \ii \R$ for all $z\in [-r, -s]$. Also, $F_+(z) = F_-(z)>0$ for all $z\in (-s, \infty)$. Consequently, the function $F$ admits analytic continuation to $\CC \backslash (-\infty, -s]$.
\end{itemize}
\end{corollary}
\begin{proof}
Define $\tau_j = \mu_j/(\mu_0^2+\ldots+\mu_d^2)$ for $j=0,\ldots, d$. Then, $\tau_0^2 + \ldots +\tau_d^2 = 1/(\mu_0^2+\ldots + \mu_d^2) = s$ and $\kappa_0 -s = s/(\mu_0^2 s-1) = -1/(\mu_1^2 + \ldots + \mu_d^2) = -r$. Observe that $F_+$ and $F_-$ are continuous on $[-r,\infty)$ since $-\mu_j^{-2} < -r$ for all $j\in \{0,\ldots, d\}$.  To prove~(a), let $\kappa \downarrow 0$ in~\eqref{eq:Vol_eq_H_eq_F_+_F_-_2}. The limit of the left-hand-side is  $\Vol_{d,0}(Q)/\omega_{d+1}$, by monotone convergence. It follows that the limit of the other terms in~\eqref{eq:Vol_eq_H_eq_F_+_F_-_2} must be finite, which implies $F_+(-s) = F_-(-s) = 0$.
Claims~(b) and (c) follow from~\eqref{eq:Vol_eq_H_eq_F_+_F_-_1} and~\eqref{eq:Vol_eq_H_eq_F_+_F_-_2}.
\end{proof}

\section*{Acknowledgement}
Supported by the German Research Foundation under Germany's Excellence Strategy  EXC 2044 -- 390685587, Mathematics M\"unster: Dynamics - Geometry - Structure and by the DFG priority program SPP 2265 Random Geometric Systems.

\addcontentsline{toc}{section}{References}
\bibliography{volume_hyperbolic_simplices_bib}
\bibliographystyle{plainnat}

\vspace{1cm}

\footnotesize

\textsc{Zakhar Kabluchko: Institut f\"ur Mathematische Stochastik,
	Universit\"at M\"unster,
	Orl\'eans-Ring 10,
	48149 M\"unster, Germany}\\
\textit{E-mail}: \texttt{zakhar.kabluchko@uni-muenster.de}\\

\textsc{Philipp Schange: Institut f\"ur Mathematische Stochastik, Universit\"at M\"unster,
	Orl\'eans-Ring 10,
	48149 M\"unster, Germany}\\
\textit{E-mail}: \texttt{philipp.schange@uni-muenster.de}

\end{document}